\date{} 
\title{Imaginary geometry III:\\ reversibility of $\SLE_\kappa$ for $\kappa \in (4,8)$}
\author{Jason Miller and Scott Sheffield}
\documentclass[12pt,naturalnames]{article}
\usepackage{amsmath}
\usepackage{amssymb}
\usepackage{amsthm}
\usepackage{amsfonts}
\usepackage{graphicx}
\usepackage{tabularx}
\usepackage{caption}
\usepackage{enumerate}
\usepackage[normalem]{ulem}
\usepackage{microtype}
\input epsf.tex

\allowdisplaybreaks

\newif\ifhyper\IfFileExists{hyperref.sty}{\hypertrue}{\hyperfalse}

\ifhyper\usepackage{hyperref}\fi

\newif\ifdraft
\drafttrue
\def\note#1/{\ifdraft {\bf [#1]}\fi}
\long\def\comment#1{}
\numberwithin{equation}{section}
\numberwithin{figure}{section}

\newtheorem{theorem}{Theorem}
\numberwithin{theorem}{section}

\newtheorem{lemma}[theorem]{Lemma}

\theoremstyle{remark}
\theoremstyle{remark}\newtheorem{remark}[theorem]{Remark}

\def\@rst #1 #2other{#1}
\newcommand\MR[1]{\relax\ifhmode\unskip\spacefactor3000 \space\fi
  \MRhref{\expandafter\@rst #1 other}{#1}}
\newcommand{\MRhref}[2]{\href{http://www.ams.org/mathscinet-getitem?mr=#1}{MR#2}}

\usepackage{subfigure}
\usepackage{graphicx}

\newcommand{\C}{\mathbf{C}}

\newcommand{\E}{\mathbf{E}}

\newcommand{\N}{\mathbf{N}}

\newcommand{\p}{\mathbf{P}}

\newcommand{\R}{\mathbf{R}}

\newcommand{\h}{\mathbf{H}}

\newcommand{\CC}{\mathcal {C}}
\newcommand{\CD}{\mathcal {D}}

\newcommand{\CI}{\mathcal {I}}

\newcommand{\CK}{\mathcal {K}}

\newcommand{\CR}{\mathcal {R}}
\newcommand{\CS}{\mathcal {S}}

\newcommand{\diam}{{\rm diam}}

\newcommand{\SLE}{{\rm SLE}}
\newcommand{\CLE}{{\rm CLE}}

\newcommand{\strip}{\CS}
\newcommand{\striptop}{\partial_U \CS}
\newcommand{\stripbot}{\partial_L \CS}

\newcommand{\reflection}{\mathcal {R}}

\newcommand{\wh}{\widehat}
\newcommand{\wt}{\widetilde}
\newcommand{\ol}{\overline}
\newcommand{\ul}{\underline}

\def\diam{\mathop{\mathrm{diam}}}

\def\Ito/{It\^o}

\def \E {{\bf E}}

\newcommand{\propFP}{Proposition~7.32} 
\newcommand{\propOB}{Proposition~7.30} 

\begin{document}
\maketitle

\begin{abstract}
Suppose that $D \subseteq \C$ is a Jordan domain and $x,y \in \partial D$ are distinct.  Fix $\kappa \in (4,8)$ and let $\eta$ be an $\SLE_\kappa$ process from $x$ to $y$ in $D$.  We prove that the law of the time-reversal of $\eta$ is, up to reparameterization, an $\SLE_\kappa$ process from $y$ to $x$ in $D$.  More generally, we prove that $\SLE_\kappa(\rho_1;\rho_2)$ processes are reversible if and only if both $\rho_i$ are at least $\kappa/2-4$, which is the critical threshold at or below which such curves are boundary filling.

Our result supplies the missing ingredient needed to show that for all $\kappa \in (4,8)$ the so-called {\em conformal loop ensembles $\CLE_\kappa$} are canonically defined, with almost surely continuous loops.  It also provides an interesting way to couple two Gaussian free fields (with different boundary conditions) so that their difference is piecewise constant and the boundaries between the constant regions are $\SLE_\kappa$ curves.
\end{abstract}

\newpage
\tableofcontents
\newpage

\medbreak {\noindent\bf Acknowledgments.}  We thank David Wilson and Dapeng Zhan for helpful discussions.

\section{Introduction}

Fix $\kappa \in (2,4)$ and write $\kappa' = 16/\kappa \in (4,8)$.  Our main result is the following:

\begin{theorem}
\label{thm::reversible}
Suppose that $D$ is a Jordan domain and let $x,y \in \partial D$ be distinct.  Let $\eta'$ be a chordal $\SLE_{\kappa'}$ process in $D$ from $x$ to $y$.  Then the law of $\eta'$ has time-reversal symmetry.  That is, if $\psi \colon D \to D$ is an anti-conformal map which swaps $x$ and $y$, then the time-reversal of $\psi \circ \eta'$ is equal in law to $\eta'$, up to reparameterization.
\end{theorem}

Since chordal $\SLE_\kappa$ curves were introduced by Schramm in 1999 \cite{S0}, they have been widely believed and conjectured to be time-reversible for all $\kappa \leq 8$.  For certain $\kappa$ values, this follows from the fact that $\SLE_\kappa$ is a scaling limit of a discrete model that does not distinguish between paths from $x$ to $y$ and paths from $y$ to $x$ ($\kappa = 2$: chordal loop-erased random walk \cite{LSW04}, $\kappa=3$: Ising model spin cluster boundaries \cite{S07}, $\kappa=4$: level lines of the discrete Gaussian free field \cite{SS09}, $\kappa=16/3$: the FK-Ising model cluster boundaries \cite{S07}, $\kappa=6$: critical percolation \cite{S01,CN06}, $\kappa=8$ uniform spanning tree boundary \cite{LSW04}).

The reversibility of chordal $\SLE_{\kappa}$ curves for arbitrary $\kappa \in (0,4]$ was established by Zhan \cite{Z_R_KAPPA} in a landmark work that builds on the commutativity approach proposed by Dub\'edat \cite{DUB_COMM} and by Schramm \cite{S0} in order to show that it is possible to construct a coupling of two $\SLE_\kappa$ curves growing at each other in opposite directions so that their ranges are almost surely equal.  By expanding on this approach, Dub\'edat \cite{DUB_DUAL} and Zhan \cite{Z_R_KAPPA_RHO} extended this result to include one-sided $\SLE_\kappa(\rho)$ processes with $\kappa \in (0,4]$ which do not intersect the boundary (i.e., $\rho \geq \tfrac{\kappa}{2}-2$).  The reversibility of the entire class of chordal $\SLE_\kappa(\rho_1;\rho_2)$ processes for $\rho_1,\rho_2 > -2$ (even when they intersect the boundary) was proved in \cite{MS_IMAG2} using a different approach, based on coupling $\SLE$ with the Gaussian free field \cite{MS_IMAG}.

This work is a sequel to and makes heavy use of the techniques and results from \cite{MS_IMAG, MS_IMAG2}.  We summarize these results in Section~\ref{subsec::imaginary}, so that this work can be read independently.  Of particular importance is a variant of the ``light cone'' characterization of $\SLE_{\kappa'}$ traces given in \cite{MS_IMAG}, which is a refinement of so-called Duplantier duality.  This gives a description of the outer boundary of an $\SLE_{\kappa'}$ process stopped upon hitting the boundary in terms of a certain $\SLE_\kappa$ process.  We will also employ the almost sure continuity of so-called $\SLE_\kappa(\ul{\rho})$ and $\SLE_{\kappa'}(\ul{\rho}')$ traces (see Section~\ref{subsec::imaginary}), even when they interact non-trivially with the boundary \cite[Theorem~1.3]{MS_IMAG}.

Theorem~\ref{thm::reversible} is a special case of a more general theorem which gives the time-reversal symmetry of $\SLE_{\kappa'}(\rho_1;\rho_2)$ processes provided $\rho_1,\rho_2 \geq \tfrac{\kappa'}{2} - 4$.  We remark that the value $\tfrac{\kappa'}{2} - 4$ is the critical threshold at or below which such processes are boundary filling \cite{MS_IMAG}.

\begin{theorem}
\label{thm::sle_kappa_rho_reversible}
Suppose that $D$ is a Jordan domain and let $x,y \in \partial D$ be distinct.  Suppose that $\eta'$ is a chordal $\SLE_{\kappa'}(\rho_1;\rho_2)$ process in $D$ from $x$ to $y$ where the force points are located at $x^-$ and $x^+$.  If $\psi \colon D \to D$ is an anti-conformal map which swaps $x$ and $y$, then the time-reversal of $\psi \circ \eta'$ is an $\SLE_{\kappa'}(\rho_1;\rho_2)$ process from $x$ to $y$, up to reparameterization.
\end{theorem}

Theorem~\ref{thm::sle_kappa_rho_reversible} has many consequences for $\SLE$.  For example, the {\em conformal loop ensembles} $\CLE_\kappa$ are random collections of loops in a planar domain, defined for all $\kappa \in (8/3,8]$.  Each loop in a $\CLE_\kappa$ looks locally like $\SLE_\kappa$, and the collection of loops can be constructed using a branching form of $\SLE_\kappa(\kappa-6)$ that traces through all of the loops, as described in \cite{SHE_CLE}.  However, there is some arbitrariness in the construction given in \cite{SHE_CLE}: one has to choose a boundary point at which to start this process, and it was not clear in \cite{SHE_CLE} whether the law of the final loop collection was independent of this initial choice; also, each loop is traced from a specific starting/ending point, and it was not clear that the ``loops'' thus constructed were actually continuous at this point.

For $\kappa \in (4,8]$ the continuity and initial-point independence were proved in \cite{SHE_CLE} as results {\em contingent} on the continuity and time-reversal symmetry of $\SLE_\kappa$ and $\SLE_\kappa(\kappa-6)$ processes.  As mentioned above, continuity was recently established in \cite{MS_IMAG}; thus Theorem~\ref{thm::sle_kappa_rho_reversible} implies that the $\CLE_\kappa$ defined in  \cite{SHE_CLE} are almost surely ensembles of continuous loops and that their laws are indeed canonical (independent of the location at which the branching form of $\SLE_\kappa(\kappa-6)$ is started).  We remark that the analogous fact for $\CLE_\kappa$ with $\kappa \in (8/3,4]$ was only recently proved in \cite{SHE_WER_CLE}.  In that case, the continuity and initial-point independence are established by showing that the branching $\SLE_\kappa(\kappa-6)$ construction of $\CLE_\kappa$ is equivalent to the {\em loop-soup-cluster-boundary} construction proposed by Werner.

Our final result is the non-reversibility of $\SLE_{\kappa'}(\rho_1;\rho_2)$ processes when either $\rho_1 < \tfrac{\kappa'}{2}-4$ or $\rho_2 < \tfrac{\kappa'}{2}-4$:

\begin{theorem}
\label{thm::sle_kappa_rho_non_reversible}
Suppose that $D$ is a Jordan domain and let $x,y \in \partial D$ be distinct.  Suppose that $\eta'$ is a chordal $\SLE_{\kappa'}(\rho_1;\rho_2)$ process in $D$ from $x$ to $y$.  Let $\psi \colon D \to D$ be an anti-conformal map which swaps $x$ and $y$.  If either $\rho_1 < \tfrac{\kappa'}{2}-4$ or $\rho_2 < \tfrac{\kappa'}{2}-4$, then the law of the time-reversal of $\psi(\eta')$ is not an $\SLE_{\kappa'}(\ul{\rho})$ process for any collection of weights $\ul{\rho}$.
\end{theorem}

\begin{figure}[h!]
\begin{center}
\includegraphics[scale=0.85]{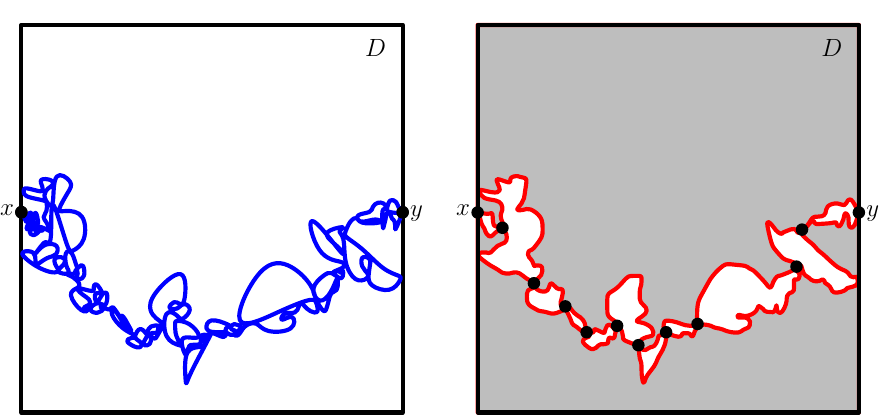}
\end{center}
\caption{\label{fig::outerboundaryreversibility}
The curve on the left represents an $\SLE_{\kappa'}(\rho_1; \rho_2)$ where $\kappa' \in (4,8)$ and $\rho_1, \rho_2 \geq \tfrac{\kappa'}{2}-4$.  It was shown in \cite[Theorem~1.4]{MS_IMAG} and \cite[Theorem~1.1]{MS_IMAG2} that the law of the outer boundary of this path (the pair of red curves from $x$ to $y$ on the right) has time-reversal symmetry; thus one can couple an $\SLE_{\kappa'}(\rho_1; \rho_2)$ path $\eta'$ from $x$ to $y$ with an $\SLE_{\kappa'}(\rho_2;\rho_1)$ path $\gamma'$ from $y$ to $x$ in such a way that their boundaries almost surely agree.  Moreover, it was also shown in \cite[\propOB]{MS_IMAG} that {\em given} these outer boundaries, the conditional law of the path within each of the white ``bubbles'' shown on the right (i.e., each of the countably many components of the complement of the boundary that lies between the two boundary paths) is given by an independent $\SLE_{\kappa'}(\tfrac{\kappa'}{2}-4; \tfrac{\kappa'}{2}-4)$ process from its first to its last endpoint (illustrated by the black dots on the right).  Thus, if $\SLE_{\kappa'}(\tfrac{\kappa'}{2}-4; \tfrac{\kappa'}{2}-4)$ has time-reversal symmetry, then we can couple $\eta'$ and $\gamma'$ so that they agree (up to time-reversal) within each bubble as well.}
\end{figure}

\subsection*{Outline}

The remainder of this article is structured as follows.  In Section~\ref{sec::preliminaries}, we will give a brief overview of both $\SLE$ and the so-called imaginary geometry of the Gaussian free field.  The latter is a non-technical summary of the results proved in \cite{MS_IMAG} which are needed for this article.  This section is similar to \cite[Section 2]{MS_IMAG2}.  In Section~\ref{sec::proofs}, we will prove Theorems~\ref{thm::reversible}-\ref{thm::sle_kappa_rho_non_reversible}.  Finally, in Section~\ref{sec::couplings} we briefly explain how these theorems can be used to construct couplings of different Gaussian free field instances with different boundary conditions; as an application, we compute a simple formula for the probability that a given point lies to the left of an $\SLE_{\kappa'}(\tfrac{\kappa'}{2}-4; \tfrac{\kappa'}{2}-4)$ curve.  (The analogous result for $\SLE_{\kappa'}$, computed by Schramm in \cite{MR1871700}, does not have such a simple form.)

As Figure~\ref{fig::outerboundaryreversibility} illustrates, when $\kappa' \in (4,8)$ and $\rho_1, \rho_2 \geq \tfrac{\kappa'}{2}-4$, the results obtained in \cite{MS_IMAG, MS_IMAG2} reduce the problem of showing time-reversal symmetry to the special case that $\eta'$ is an $\SLE_{\kappa'}(\tfrac{\kappa'}{2}-4;\tfrac{\kappa'}{2}-4)$, which is a random curve that hits every point on the entire boundary almost surely.  The second step is to pick some point $z$ on the boundary of $D$ and consider the outer boundaries of the {\em past} and {\em future} of $\eta'$ upon hitting $z$ --- i.e., the outer boundary of the set of points visited by $\eta'$ before hitting $z$ and the outer boundary of the set of points visited by $\eta'$ after hitting $z$, as illustrated in Figure~\ref{fig::iterativesplit}.  Lemma~\ref{lem::reflecting_strip} shows that the law of this pair of paths is invariant under the anti-conformal map $D \to D$ that swaps $x$ and $y$ while fixing $z$.

The proof of Lemma~\ref{lem::reflecting_strip} is the heart of the argument.  It makes use of Gaussian free field machinery in a rather picturesque way that avoids the need for extensive calculations.  Roughly speaking, we will first consider an ``infinite volume limit'' obtained by ``zooming in'' near the point $z$ in Figure~\ref{fig::iterativesplit}.  In this limit, we find a coupled pair of $\SLE_\kappa(\rho_1; \rho_2)$ paths from $z \in \partial \h$ to $\infty$ in $\h$ and \cite[Theorem~1.1]{MS_IMAG2} implies that the law of the pair of paths is invariant under reflection about the vertical axis through $z$.  By employing a second trick (involving a second pair of paths started at a second point in $\partial \h$) we are able to recover the finite volume symmetry from the infinite volume symmetry.

\begin{figure}[h!]
\begin{center}
\includegraphics[scale=0.85]{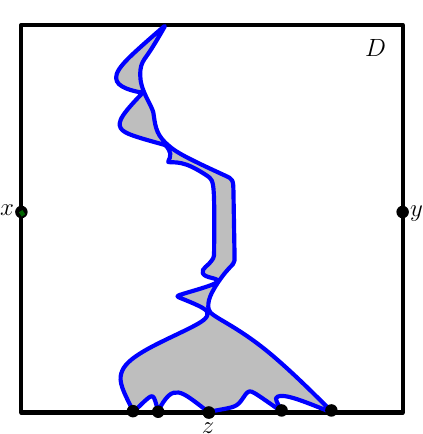}
\end{center}
\caption{\label{fig::iterativesplit}
Let $\eta'$ be an $\SLE_{\kappa'}(\tfrac{\kappa'}{2}-4;\tfrac{\kappa'}{2}-4)$ (which almost surely hits every point on $\partial D$) and let $z$ be a fixed point on $\partial D$.  The left of the two blue curves shown (starting at $z$, ending at the top of the box) is the outer boundary of the ``past of $z$'' (i.e., the set of all points $\eta'$ disconnects from $y$ before $z$ is hit).  The right blue curve with the same endpoints is the outer boundary of the ``future of $z$'' (i.e., the set of all points the time-reversal of $\eta'$ disconnects from $x$ before $z$ is hit).    Lemma~\ref{lem::reflecting_strip} shows that the law of this pair of paths is invariant under the anti-conformal map $D \to D$ that swaps $x$ and $y$ while fixing $z$.  Thus we can couple forward and reverse $\SLE_{\kappa'}(\tfrac{\kappa'}{2}-4;\tfrac{\kappa'}{2}-4)$ processes so that these boundaries are the same for both of them.  The conditional law of $\eta'$ within each of the white bubbles is given by an independent $\SLE_{\kappa'}(\tfrac{\kappa'}{2}-4;\tfrac{\kappa'}{2}-4)$ process \cite[Proposition~7.24]{MS_IMAG}, so we can iterate this construction.}
\end{figure}

Once we have Lemma~\ref{lem::reflecting_strip}, we can couple forward and reverse $\SLE_{\kappa'}(\tfrac{\kappa'}{2}-4;\tfrac{\kappa'}{2}-4)$ processes so that their past and future upon hitting $z$ have the same outer boundaries.  Moreover, given the information in Figure~\ref{fig::iterativesplit}, \cite[\propFP]{MS_IMAG} implies that the conditional law of $\eta'$ within each of these bubbles is again an independent $\SLE_{\kappa'}(\tfrac{\kappa'}{2}-4;\tfrac{\kappa'}{2}-4)$ process.  (This is a consequence of the ``light cone'' characterization of $\SLE_{\kappa'}$ processes established in \cite{MS_IMAG}.)  Thus we can pick any point on the boundary of a bubble and further couple so that the past and future of that point (within the bubble) have the same boundary.  Iterating this procedure a countably infinite number of times allows us to couple two $\SLE_{\kappa'}(\tfrac{\kappa'}{2}-4;\tfrac{\kappa'}{2}-4)$
curves so that one is almost surely the time-reversal of the other, thereby proving  Theorem~\ref{thm::sle_kappa_rho_reversible}.  The non-reversibility when one of $\rho_1$ or $\rho_2$ is less than $\tfrac{\kappa'}{2}-4$ is shown by checking that the analog of Figure~\ref{fig::iterativesplit} is {\em not} invariant under anti-conformal maps (fixing $z$, swapping $x$ and $y$) in this case.

\section{Preliminaries}
\label{sec::preliminaries}

The purpose of this section is to review the basic properties of $\SLE_\kappa(\ul{\rho}^L;\ul{\rho}^R)$ processes in addition to giving a non-technical overview of the so-called imaginary geometry of the Gaussian free field.  The latter is a mechanism for constructing couplings of many $\SLE_\kappa(\ul{\rho}^L;\ul{\rho}^R)$ strands in such a way that it is easy to compute the conditional law of one of the curves given the realization of the others \cite{MS_IMAG}.

\subsection{$\SLE_\kappa(\rho)$ Processes}

$\SLE_\kappa$ is a one-parameter family of conformally invariant random curves, introduced by Oded Schramm in \cite{S0} as a candidate for (and later proved to be) the scaling limit of loop erased random walk \cite{LSW04} and the interfaces in critical percolation \cite{S01, CN06}.  Schramm's curves have been shown so far also to arise as the scaling limit of the macroscopic interfaces in several other models from statistical physics: \cite{SS09,S07,CS10U,SS05,MillerSLE}.  More detailed introductions to $\SLE$ can be found in many excellent survey articles of the subject, e.g., \cite{W03, LAW05}.

An $\SLE_\kappa$ in $\h$ from $0$ to $\infty$ is defined by the random family of conformal maps $g_t$ obtained by solving the Loewner ODE
\begin{equation}
\label{eqn::loewner_ode}
\partial_t g_t(z) = \frac{2}{g_t(z) - W_t},\ \ \ g_0(z) = z
\end{equation}
where $W = \sqrt{\kappa} B$ and $B$ is a standard Brownian motion.  Write $K_t := \{z \in \h: \tau(z) \leq t \}$.  Then $g_t$ is a conformal map from $\h_t := \h \setminus K_t$ to $\h$ satisfying $\lim_{|z| \to \infty} |g_t(z) - z| = 0$.

Rohde and Schramm showed that there almost surely exists a curve $\eta$ (the so-called $\SLE$ \emph{trace}) such that for each $t \geq 0$ the domain $\h_t$ is the unbounded connected component of $\h \setminus \eta([0,t])$, in which case the (necessarily simply connected and closed) set $K_t$ is called the ``filling'' of $\eta([0,t])$ \cite{RS05}.  An $\SLE_\kappa$ connecting boundary points $x$ and $y$ of an arbitrary simply connected Jordan domain can be constructed as the image of an $\SLE_\kappa$ on $\h$ under a conformal transformation $\psi \colon \h \to D$ sending $0$ to $x$ and $\infty$ to $y$.  (The choice of $\psi$ does not affect the law of this image path, since the law of $\SLE_\kappa$ on $\h$ is scale invariant.)  $\SLE_\kappa$ is characterized by the fact that it satisfies the domain Markov property and is invariant under conformal transformations.

$\SLE_\kappa(\ul{\rho}^L;\ul{\rho}^R)$ is the stochastic process one obtains by solving~\eqref{eqn::loewner_ode} where the driving function $W$ is taken to be the solution to the SDE
\begin{equation}
\label{eqn::sle_kappa_rho_eqn}
\begin{split}
dW_t &= \sqrt{\kappa} dB_t + \sum_{q \in \{L,R\}} \sum_{i} \frac{\rho^{i,q}}{W_t - V_t^{i,q}} dt \\
dV_t^{i,q} &= \frac{2}{V_t^{i,q} - W_t} dt,\quad V_0^{i,q} = x^{i,q}.
\end{split}
\end{equation}
Like $\SLE_\kappa$, the $\SLE_\kappa(\ul{\rho}^L;\ul{\rho}^R)$ processes arise in a variety of natural contexts.  The existence and uniqueness of solutions to~\eqref{eqn::sle_kappa_rho_eqn} is discussed in \cite[Section 2]{MS_IMAG}.  In particular, it is shown that there is a unique solution to~\eqref{eqn::sle_kappa_rho_eqn} until the first time $t$ that $W_t = V_t^{j,q}$ where $\sum_{i=1}^j \rho^{i,q} \leq -2$ for $q \in \{L,R\}$ --- we call this time the {\bf continuation threshold} (see \cite[Section 2]{MS_IMAG}).  In particular, if $\sum_{i=1}^j \rho^{i,q} > -2$ for all $1 \leq j \leq |\ul{\rho}^q|$ for $q \in \{L,R\}$, then~\eqref{eqn::sle_kappa_rho_eqn} has a unique solution for all times $t$.  This even holds when one or both of the $x^{1,q}$ are zero.  The almost sure continuity of the $\SLE_\kappa(\ul{\rho}^L;\ul{\rho}^R)$ trace is also proved in \cite[Theorem~1.3]{MS_IMAG}.

\subsection{Imaginary Geometry of the Gaussian Free Field}
\label{subsec::imaginary}

We will now give an overview of the so-called \emph{imaginary geometry} of the Gaussian free field (GFF).  In this article, this serves as a tool for constructing couplings of multiple $\SLE$ strands and provides a simple calculus for computing the conditional law of one of the strands given the realization of the others \cite{MS_IMAG}.  The purpose of this overview is to explain just enough of the theory so that this article may read and understood independently of \cite{MS_IMAG}, however we refer the reader interested in proofs of the statements we make here to \cite{MS_IMAG}.  We begin by fixing a domain $D \subseteq \C$ with smooth boundary and letting $C_0^\infty(D)$ denote the space of compactly supported $C^\infty$ functions on $D$.  For $f,g \in C_0^\infty(D)$, we let
\[ (f,g)_\nabla := \frac{1}{2\pi} \int_D \nabla f (x) \cdot \nabla g(x)dx\]
denote the Dirichlet inner product of $f$ and $g$ where $dx$ is the Lebesgue measure on $D$.  Let $H(D)$ be the Hilbert space closure of $C_0^\infty(D)$ under $(\cdot,\cdot)_\nabla$.  The continuum Gaussian free field $h$ (with zero boundary conditions) is the so-called standard Gaussian on $H(D)$.  It is given formally as a random linear combination
\begin{equation}
\label{eqn::gff_definition}
 h = \sum_n \alpha_n f_n,
\end{equation}
where $(\alpha_n)$ are i.i.d.\ $N(0,1)$ and $(f_n)$ is an orthonormal basis of $H(D)$.  The GFF with non-zero boundary data $\psi$ is given by adding the harmonic extension of $\psi$ to a zero-boundary GFF $h$.

The GFF is a two-dimensional-time analog of Brownian motion.  Just as Brownian motion can be realized as the scaling limit of many random lattice walks, the GFF arises as the scaling limit of many random (real or integer valued) functions on two dimensional lattices \cite{BAD96, KEN01, NS97, RV08, MillerGLCLT}.  The GFF can be used to generate various kinds of random geometric structures, in particular the imaginary geometry discussed here \cite{2010arXiv1012.4797S, MS_IMAG}.  This corresponds to considering $e^{ih/\chi}$, for a fixed constant $\chi >0$.  Informally, the ``rays'' of the imaginary geometry are flow lines of the complex vector field $e^{i(h/\chi+\theta)}$, i.e., solutions to the ODE \begin{equation}
\eta'(t) = e^{i \left(h(\eta(t))+\theta\right)} \text{ for } t > 0,
\end{equation}
for given values of $\eta(0)$ and $\theta$.

\begin{figure}[h!]
\begin{center}
\includegraphics[scale=0.85]{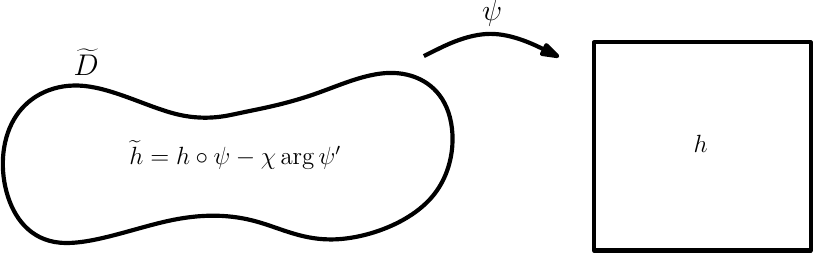}
\caption{\label{fig::coordinatechange} The set of flow lines in $\wt{D}$ is the pullback via a conformal map $\psi$ of the set of flow lines in $D$ provided $h$ is transformed to a new function $\wt{h}$ in the manner shown.}
\end{center}
\end{figure}

\begin{figure}[h!]
\begin{center}
\subfigure{\includegraphics[scale=0.85]{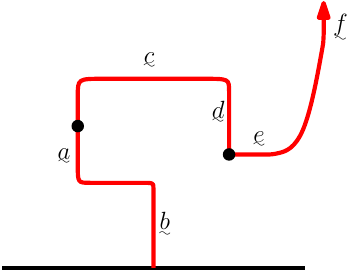}}
\hspace{0.01\textwidth}
\subfigure{\includegraphics[scale=0.85]{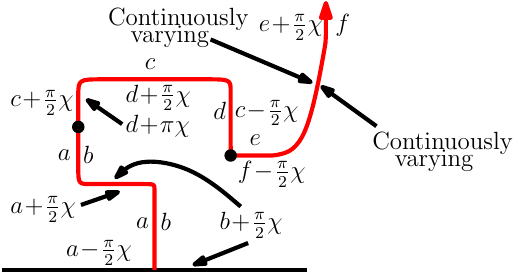}}
\caption{\label{fig::winding}  We will often make use of the notation depicted on the left hand side to indicate boundary values for Gaussian free fields.  Specifically, we will delineate the boundary $\partial D$ of a Jordan domain $D$ with black dots.  On each arc $L$ of $\partial D$ which lies between a pair of black dots, we will draw either a horizontal or vertical segment $L_0$ and label it with $\uwave{x}$.  This means that the boundary data on $L_0$ is given by $x$.  Whenever $L$ makes a quarter turn to the right, the height goes down by $\tfrac{\pi}{2} \chi$ and whenever $L$ makes a quarter turn to the left, the height goes up by $\tfrac{\pi}{2} \chi$.  More generally, if $L$ makes a turn which is not necessarily at a right angle, the boundary data is given by $\chi$ times the winding of $L$ relative to $L_0$.  If we just write $x$ next to a horizontal or vertical segment, we mean to indicate the boundary data just at that segment and nowhere else.  The right side above has exactly the same meaning as the left side, but the boundary data is spelled out explicitly everywhere.  Even when the curve has a fractal, non-smooth structure, the {\em harmonic extension} of the boundary values still makes sense, since one can transform the figure via the rule in Figure~\ref{fig::coordinatechange} to a half plane with piecewise constant boundary conditions. }
\end{center}
\end{figure}

A brief overview of imaginary geometry (as defined for general functions $h$) appears in \cite{2010arXiv1012.4797S}, where the rays are interpreted as geodesics of a variant of the Levi-Civita connection associated with Liouville quantum gravity.  One can interpret the $e^{ih}$ direction as ``north'' and the $e^{i(h + \pi/2)}$ direction as ``west'', etc.  Then $h$ determines a way of assigning a set of compass directions to every point in the domain, and a ray is determined by an initial point and a direction.  When $h$ is constant, the rays correspond to rays in ordinary Euclidean geometry.  For more general continuous $h$, one can still show that when three rays form a triangle, the sum of the angles is always $\pi$ \cite{2010arXiv1012.4797S}.

To build these rays, one begins by constructing explicit couplings of $h$ with variants of $\SLE$ and showing that these couplings have certain properties.  Namely, if one conditions on part of the curve, then the conditional law of $h$ is that of a GFF in the complement of the curve with certain boundary conditions (see Figure~\ref{fig::conditional_boundary_data}).  Examples of these couplings appear in \cite{She_SLE_lectures, SchrammShe10, DUB_PART, 2010arXiv1012.4797S} as well as variants in \cite{MakarovSmirnov09,HagendorfBauerBernard10,IzyurovKytola10}.  The next step is to show that in these couplings the path is almost surely \emph{determined by the field} so that we can really interpret the ray as a path-valued function of the field.  This step is carried out in some generality in \cite{DUB_PART, 2010arXiv1012.4797S, MS_IMAG}.

If $h$ is a smooth function, $\eta$ a flow line of $e^{ih/\chi}$, and $\psi \colon \wt D \to D$ a conformal transformation, then by the chain rule, $\psi^{-1}(\eta)$ is a flow line of $h \circ \psi - \chi \arg \psi'$, as in Figure~\ref{fig::coordinatechange}. With this in mind, we define an {\bf imaginary surface} to be an equivalence class of pairs $(D,h)$ under the equivalence relation
\begin{equation}
\label{eqn::ac_eq_rel}
 (D,h) \rightarrow (\psi^{-1}(D), h \circ \psi - \chi \arg \psi') = (\wt{D},\wt{h}).
\end{equation}
We interpret $\psi$ as a (conformal) {\em coordinate change} of the imaginary surface.  In what follows, we will generally take $D$ to be the upper half plane, but one can map the flow lines defined there to other domains using~\eqref{eqn::ac_eq_rel}.

\begin{figure}[h!]
\begin{center}
\includegraphics[scale=0.85]{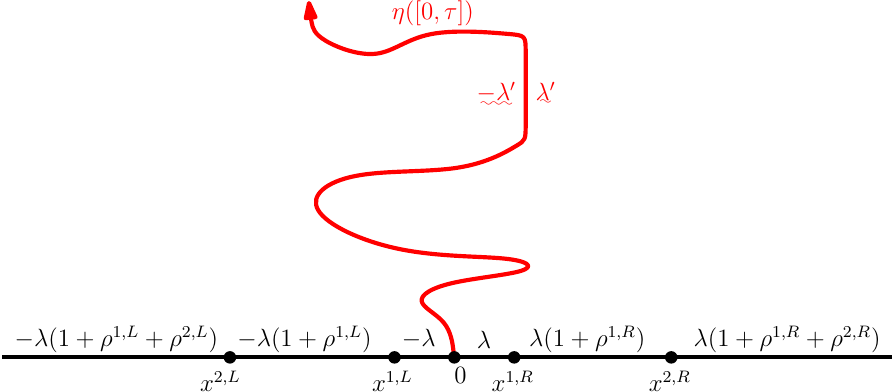}
\caption{\label{fig::conditional_boundary_data}  Suppose that $h$ is a GFF on $\h$ with the boundary data depicted above.  Then the flow line $\eta$ of $h$ starting from $0$ is an $\SLE_\kappa(\ul{\rho}^L;\ul{\rho}^R)$ curve in $\h$ where $|\ul{\rho}^L| = |\ul{\rho}^R| = 2$.  Conditional on $\eta([0,\tau])$ for any $\eta$ stopping time $\tau$, $h$ is equal in distribution to a GFF on $\h \setminus \eta([0,\tau])$ with the boundary data on $\eta([0,\tau])$ depicted above (the notation $\uwave{a}$ which appears adjacent to $\eta([0,\tau])$ is explained in some detail in Figure~\ref{fig::winding}).  It is also possible to couple $\eta' \sim\SLE_{\kappa'}(\ul{\rho}^L;\ul{\rho}^R)$ for $\kappa' > 4$ with $h$ and the boundary data takes on the same form.  The difference is in the interpretation --- $\eta'$ is not a flow line of $h$, but for each time $\tau$, the left and right outer boundaries of the set $\eta'([0,\tau])$, traced starting from $\eta(\tau)$, are flow lines of $h$ with appropriate angles.}
\end{center}
\end{figure}

We assume throughout the rest of this section that $\kappa \in (0,4)$ so that $\kappa' := 16/\kappa \in (4,\infty)$.  When following the illustrations, it will be useful to keep in mind a few definitions and identities:
\begin{equation} \label{eqn::deflist} \lambda := \frac{\pi}{\sqrt \kappa}, \,\,\,\,\,\,\,\,\lambda' := \frac{\pi}{\sqrt{16/\kappa}} = \frac{\pi \sqrt{\kappa}}{4} = \frac{\kappa}{4} \lambda < \lambda, \,\,\,\,\,\,\,\, \chi := \frac{2}{\sqrt \kappa} - \frac{\sqrt \kappa}{2} > 0\end{equation}
\begin{equation} \label{eqn::fullrevolution} 2 \pi \chi = 4(\lambda-\lambda'), \,\,\,\,\,\,\,\,\,\,\,\lambda' = \lambda - \frac{\pi}{2} \chi \end{equation}
\begin{equation} \label{eqn::fullrevolutionrho}
2 \pi \chi = (4-\kappa)\lambda = (\kappa'-4)\lambda'.
\end{equation}

\begin{figure}[h!]
\begin{center}
\includegraphics[scale=0.85]{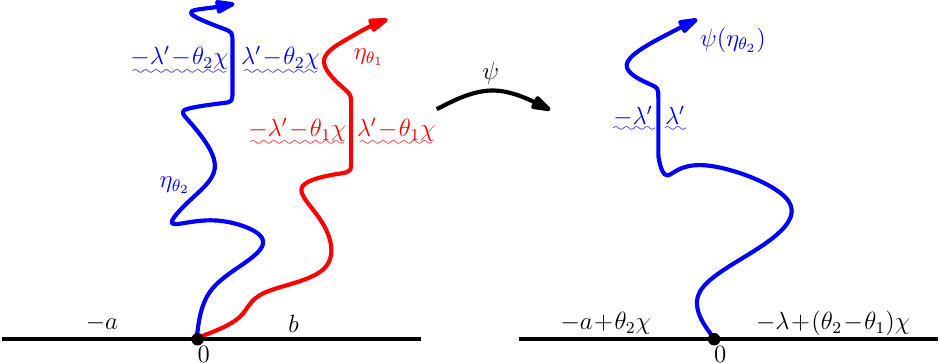}
\caption{\label{fig::monotonicity}  Suppose that $h$ is a GFF on $\h$ with the boundary data depicted above.  For each $\theta \in \R$, let $\eta_\theta$ be the flow line of the GFF $h+\theta \chi$.  This corresponds to setting the angle of $\eta_\theta$ to be $\theta$.  Just as if $h$ were a smooth function, if $\theta_1 < \theta_2$ then $\eta_{\theta_1}$ lies to the right of $\eta_{\theta_2}$ \cite[Theorem~1.5]{MS_IMAG}.  The conditional law of $h$ given $\eta_{\theta_1}$ and $\eta_{\theta_2}$ is a GFF on $\h \setminus \cup_{i=1}^2 \eta_{\theta_i}$ whose boundary data is shown above \cite[Proposition~6.1]{MS_IMAG}.  By applying a conformal mapping and using the transformation rule~\eqref{eqn::ac_eq_rel}, we can compute the conditional law of $\eta_{\theta_1}$ given the realization of $\eta_{\theta_2}$ and vice-versa.  That is, $\eta_{\theta_2}$ given $\eta_{\theta_1}$ is an $\SLE_\kappa((a-\theta_2 \chi)/\lambda -1; (\theta_2-\theta_1)\chi/\lambda-2)$ process independently in each of the connected components of $\h \setminus \eta_{\theta_1}$ which lie to the left of $\eta_{\theta_1}$.  Moreover, $\eta_{\theta_1}$ given $\eta_{\theta_2}$ is an $\SLE_\kappa((\theta_2-\theta_1) \chi/\lambda -2;(b+\theta_1\chi)/\lambda-1)$ independently in each of the connected components of $\h \setminus \eta_{\theta_2}$ which lie to the right of $\eta_{\theta_2}$ \cite[Section 7.1]{MS_IMAG}.}
\end{center}
\end{figure}

The boundary data one associates with the GFF on $\h$ so that its flow line from $0$ to $\infty$ is an $\SLE_\kappa(\ul{\rho}^L;\ul{\rho}^R)$ process with force points located at $\ul{x} = (\ul{x}^L,\ul{x}^R)$ is
\begin{align}
 -&\lambda\left( 1 + \sum_{i=1}^j \rho^{i,L}\right) \text{ for } x \in [x^{j+1,L},x^{j,L}) \text{ and }\\
 &\lambda\left( 1 + \sum_{i=1}^j \rho^{i,R}\right) \text{ for } x \in [x^{j,R},x^{j+1,R})
\end{align}
This is depicted in Figure~\ref{fig::conditional_boundary_data} in the special case that $|\ul{\rho}^L| = |\ul{\rho}^R| = 2$.  As we explained earlier, for any $\eta$ stopping time $\tau$, the law of $h$ conditional on $\eta([0,\tau])$ is a GFF in $\h \setminus \eta([0,\tau])$.  The boundary data of the conditional field agrees with that of $h$ on $\partial \h$.  On the right side of $\eta([0,\tau])$, it is $\lambda' + \chi \cdot {\rm winding}$, where the terminology ``winding'' is explained in Figure~\ref{fig::winding}, and to the left it is $-\lambda' + \chi \cdot {\rm winding}$.  This is also depicted in Figure~\ref{fig::conditional_boundary_data}.

By considering several flow lines of the same field, we can construct couplings of multiple $\SLE_\kappa(\ul{\rho})$ processes.  For example, suppose that $\theta \in \R$.  The flow line $\eta_\theta$ of $h+\theta\chi$ should be interpreted as the flow line of the vector field $e^{ih/\chi + \theta}$.  That is, $\eta_\theta$ is the flow line of $h$ with angle $\theta$.  If $h$ were a continuous function and we had $\theta_1 < \theta_2$, then it would be obvious that $\eta_{\theta_1}$ lies to the right of $\eta_{\theta_2}$.  Although non-trivial to prove, this is also true in the setting of the GFF \cite[Theorem~1.5]{MS_IMAG} and is depicted in Figure~\ref{fig::monotonicity}.

For $\theta_1 < \theta_2$, we can compute the conditional law of $\eta_{\theta_2}$ given $\eta_{\theta_1}$ \cite[Section 7.1]{MS_IMAG}.  It is an $\SLE_\kappa((a-\theta_2 \chi)/\lambda-1; (\theta_2-\theta_1) \chi/\lambda -2)$ process independently in each connected component of $\h \setminus \eta_{\theta_1}$ which lies to the left of $\eta_{\theta_1}$ \cite[Section 7.1]{MS_IMAG}.  Moreover, $\eta_{\theta_1}$ given $\eta_{\theta_2}$ is independently an $\SLE_\kappa((\theta_2-\theta_1) \chi/\lambda -2;(b+\theta_1\chi)/\lambda-1)$ in each of the connected components of $\h \setminus \eta_{\theta_2}$ which lie to the right of $\eta_{\theta_2}$.  This is depicted in Figure~\ref{fig::monotonicity}.

\begin{figure}[ht!]
\begin{center}
\includegraphics[scale=0.85]{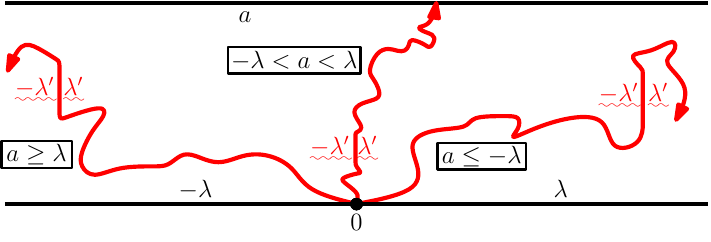}
\caption{\label{fig::hittingrange}  Suppose that $h$ is a GFF on the strip $\strip$ with the boundary data depicted above and $\eta$ is the flow line of $h$ starting at $0$.  The interaction of $\eta$ with the upper boundary $\striptop$ of $\partial \strip$ depends on $a$, the boundary data of $h$ on $\striptop$.  Curves shown represent almost sure behaviors corresponding to the three different regimes of $a$ (indicated by the closed boxes).  The path hits $\striptop$ almost surely if and only if $a \in (-\lambda, \lambda)$.  When $a \geq \lambda$, it tends to $-\infty$ (left end of the strip) and when $a \leq - \lambda$ it tends to $\infty$ (right end of the strip) without hitting $\striptop$.  This can be rephrased in terms of the weights $\ul{\rho}$: an $\SLE_{\kappa}(\ul{\rho})$ process almost surely does not hit a boundary interval $(x^{i,R},x^{i+1,R})$ (resp.\ $(x^{i+1,L},x^i)$) if $\sum_{s=1}^i \rho^{s,R} \geq \tfrac{\kappa}{2}-2$ (resp.\ $\sum_{s=1}^i \rho^{s,L} \geq \tfrac{\kappa}{2}-2$).  See \cite[Lemma~5.2]{MS_IMAG} and \cite[Remark~5.3]{MS_IMAG}.  These facts hold for all $\kappa > 0$.}
\end{center}
\end{figure}

\begin{figure}[ht!]
\begin{center}
\includegraphics[scale=0.85]{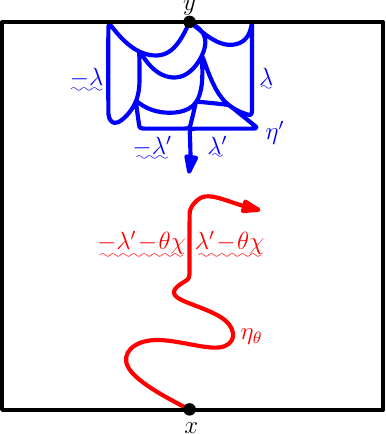}
\end{center}
\caption{\label{fig::counterflowline} We can construct $\SLE_\kappa$ flow lines and $\SLE_{\kappa'}$, $\kappa' = 16/\kappa \in (4,\infty)$, counterflow lines within the same geometry.  This is depicted above for a single counterflow line $\eta'$ emanating from $y$ and a flow line $\eta_\theta$ with angle $\theta$ starting from $0$ (we intentionally did not describe the boundary data of $h$ on $\partial D$).  If $\theta = \theta_R := \tfrac{1}{\chi}(\lambda'-\lambda) = -\tfrac{\pi}{2}$ so that the boundary data on the right side of $\eta_\theta$ matches that on the right side of $\eta'$, then $\eta_\theta$ will almost surely hit and then ``merge'' into the right boundary of $\eta'([0,\tau'])$ for any $\eta'$ stopping time $\tau'$ and, more generally, the right boundary of the entire trace of $\eta'$ is given by $\eta_\theta$ --- this fact is known as $\SLE$ duality.  Analogously, if $\theta = \theta_L := \tfrac{1}{\chi}(\lambda-\lambda') = \tfrac{\pi}{2} = -\theta_R$, then $\eta_\theta$ will almost surely hit and then merge into the left boundary of $\eta'([0,\tau'])$ and is equal to the left boundary of the entire trace of $\eta'$.  These facts are proved in \cite[Theorem~1.4]{MS_IMAG}.}
\end{figure}

It is also possible to determine which segments of the boundary a flow or counterflow line cannot hit.  This is described in terms of the boundary data of the field in Figure~\ref{fig::hittingrange} and proved in \cite[Lemma~5.2]{MS_IMAG} (this result gives the range of boundary data that $\eta$ cannot hit, contingent on the almost sure continuity of $\eta$; this, in turn, is given in \cite[Theorem~1.3]{MS_IMAG}).  This can be rephrased in terms of the weights $\ul{\rho}$: an $\SLE_{\kappa}(\ul{\rho})$ process almost surely does not hit a boundary interval $(x^{i,R},x^{i+1,R})$ (resp.\ $(x^{i+1,L},x^i)$) if $\sum_{s=1}^i \rho^{s,R} \geq \tfrac{\kappa}{2}-2$ (resp.\ $\sum_{s=1}^i \rho^{s,L} \geq \tfrac{\kappa}{2}-2$).  See \cite[Remark~5.3]{MS_IMAG}.

Recall that $\kappa' = 16/\kappa \in (4,\infty)$.  We refer to $\SLE_{\kappa'}$ processes $\eta'$ as \emph{counterflow} lines.  The left boundaries of $\eta'([0, \tau'])$, taken over a range of $\tau'$ values, form a tree structure comprised of $\SLE_\kappa$ flow lines which in some sense run orthogonal to $\eta'$.  The right boundaries form a dual tree structure.  We can construct couplings of $\SLE_{\kappa}$ and $\SLE_{\kappa'}$ processes (flow lines and counterflow lines) within the same imaginary geometry \cite[Theorem~1.4]{MS_IMAG}.  This is depicted in Figure~\ref{fig::counterflowline} in the special case of a single flow line $\eta_\theta$ with angle $\theta$ emanating from $x$ and targeted at $y$ and a single counterflow line $\eta'$ emanating from $y$.  When $\theta > \tfrac{1}{\chi}(\lambda-\lambda') = \tfrac{\pi}{2}$, $\eta_\theta$ almost surely passes to the left of (though may hit the left boundary of) $\eta'$ \cite[Theorem~1.4 and Theorem~1.5]{MS_IMAG}.  If $\theta = \tfrac{\pi}{2}$, then $\eta_\theta$ is equal to the left boundary of $\eta'$.  There is some intuition provided for this in Figure~\ref{fig::counterflowline}.  Analogously, if $\theta < \tfrac{1}{\chi}(\lambda'-\lambda) = -\tfrac{\pi}{2}$, then $\eta_{\theta}$ passes to the right of $\eta'$ \cite[Theorem~1.4 and Theorem~1.5]{MS_IMAG} and when $\theta = -\tfrac{\pi}{2}$, $\eta_\theta$ is equal to the right boundary of $\eta'$.

\begin{figure}[h!]
\begin{center}
\includegraphics[scale=0.85]{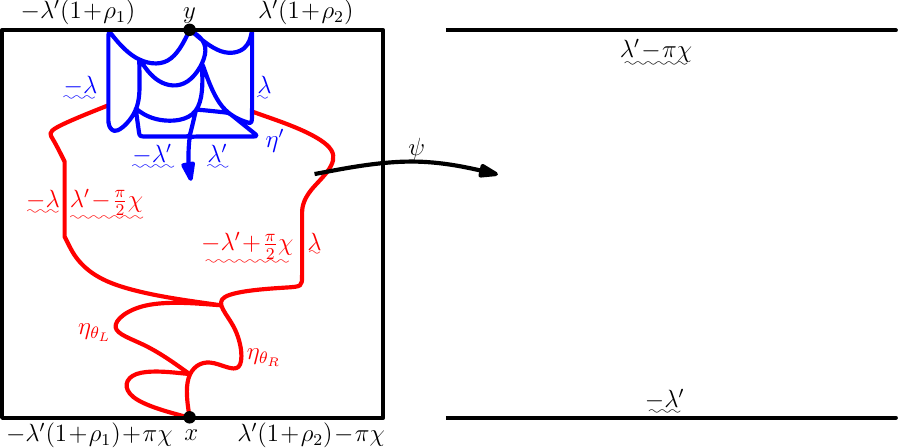}
\end{center}
\caption{\label{fig::duality} (Continuation of Figure~\ref{fig::counterflowline}).  We now assume that the boundary data for $h$ is as depicted above and that $\rho_1,\rho_2 > \tfrac{\kappa'}{2}-4$.  Then $\eta' \sim \SLE_{\kappa'}(\rho_1;\rho_2)$.  Let $\eta_{\theta_L}$ and $\eta_{\theta_R}$ be the left and right boundaries of the counterflow line $\eta'$, respectively.  One can check that in this case, $\eta_{\theta_q} \sim \SLE_{\kappa}(\rho_1^q;\rho_2^q)$ with $\rho_1^q,\rho_2^q > -2$ for $q \in \{L,R\}$ (see Figure~\ref{fig::conditional_boundary_data} and recall the transformation rule~\eqref{eqn::ac_eq_rel}).  Each connected component $C$ of $D \setminus (\eta_{\theta_L} \cup \eta_{\theta_R})$ which lies between $\eta_{\theta_L}$ and $\eta_{\theta_R}$ has two distinguished points $x_C$ and $y_C$ --- the first and last points on $\partial C$ traced by $\eta_{\theta_L}$ (as well as by $\eta_{\theta_R}$).  In each such $C$, the law of $\eta'$ is independently an $\SLE_{\kappa'}(\tfrac{\kappa'}{2}-4;\tfrac{\kappa'}{2}-4)$ process from $y_C$ to $x_C$ \cite[\propOB]{MS_IMAG}.  If we apply a conformal change of coordinates $\psi \colon C \to \strip$ with $\psi(x_C) = -\infty$ and $\psi(y_C) = \infty$, then the law of $h \circ \psi^{-1} - \chi \arg (\psi^{-1})'$ is a GFF on $\strip$ whose boundary data is depicted on the right hand side.  Moreover, $\psi(\eta')$ is the counterflow line of this field running from $+\infty$ to $-\infty$ and almost surely hits every point on $\partial \strip$.  This holds more generally whenever the boundary data is such that $\eta_{\theta_L},\eta_{\theta_R}$ make sense as flow lines of $h$ until terminating at $y$ (i.e., the continuation threshold it not hit until the process terminates at $y$).}
\end{figure}

\begin{figure}[ht!]
\begin{center}
\includegraphics[scale=0.85]{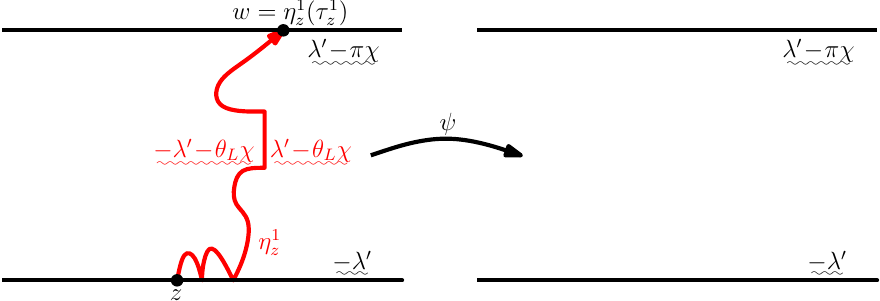}
\end{center}
\caption{\label{fig::duality2} Suppose that $h$ is a GFF on $\strip$ whose boundary data is depicted above and fix $z$ in the lower boundary $\stripbot$ of $\strip$.  Then the counterflow line $\eta'$ of $h$ from $\infty$ to $-\infty$ is an $\SLE_{\kappa'}(\tfrac{\kappa'}{2}-4;\tfrac{\kappa'}{2}-4)$ process (see Figure~\ref{fig::conditional_boundary_data} and recall the transformation rule~\eqref{eqn::ac_eq_rel}) and almost surely hits $z$, say at time $\tau_z'$.  The left boundary of $\eta'([0,\tau_z'])$ is almost surely equal to the flow line $\eta_z^1$ of $h$ starting at $z$ with angle $\theta_L = \tfrac{\pi}{2}$ stopped at time $\tau_z^1$, the first time it hits the upper boundary $\striptop$ of $\strip$.  The conditional law of $h$ given $\eta_z^1([0,\tau_z^1])$ in each connected component of $\strip \setminus \eta_z^1([0,\tau_z^1])$ which lies to the right of $\eta_z^1([0,\tau_z^1])$ is the same as $h$ itself, up to a conformal change of coordinates which preserves the entrance and exit points of $\eta'$.  The conditional law of $\eta'$ within each such component is independently that of an $\SLE_{\kappa'}(\tfrac{\kappa'}{2}-4;\tfrac{\kappa'}{2}-4)$ from the first to last endpoint \cite[\propFP]{MS_IMAG}.
}
\end{figure}

\begin{figure}[ht!]
\begin{center}
\includegraphics[scale=0.85]{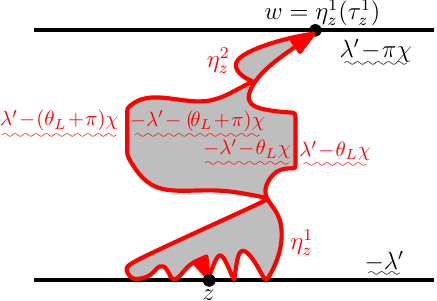}
\end{center}
\caption{\label{fig::duality3}(Continuation of Figure~\ref{fig::duality2})  Moreover, $\eta'([\tau_z',\infty))$ almost surely stays to the left of $\eta_z^1([0,\tau_z^1])$.  It is the counterflow line of $h$ restricted to the left component of $\strip \setminus \eta_z^1([0,\tau_z^1])$, starting at $z$ and running to $-\infty$ \cite[\propFP]{MS_IMAG}.  Let $w = \eta_z^1(\tau_z^1)$.  Since $\eta'$ is boundary filling and cannot enter into the loops it creates with itself or the boundary, the first point on $\striptop$ that $\eta'$ hits after time $\tau_z'$ is $w$.  The left boundary of $\eta'|_{[\tau_z',\infty)}$ is given by the flow line $\eta_z^2$ of $h$ given $\eta_z^1([0,\tau_z^1])$ in the left connected component of $\strip \setminus \eta_z^1([0,\tau_z^1])$ with angle $\theta_L$, started at $w$ and stopped at the time $\tau_z^2$ that it first hits $z$.  The order in which $\eta'$ hits the connected components of $\strip \setminus (\eta_z^1([0,\tau_z^1]) \cup \eta_z^2([0,\tau_z^2]))$ which lie to the right of $\eta_z^1([0,\tau_z^1])$ is determined by the reverse chronological order that $\eta_z^1$ traces their boundary.  Likewise, the order in which $\eta'$ hits those connected components which lie to the left of $\eta_z^2([0,\tau_z^2])$ is determined by the reverse chronological order that $\eta_z^2$ traces their boundary \cite[\propFP]{MS_IMAG}.  Lemma~\ref{lem::reflecting_strip} states that the law of the pair $\{\eta_z^1|_{[0,\tau_z^1]}, \eta_z^2|_{[0,\tau_z^2]}\}$ is invariant under reflection about the vertical axis through $z$ up to time-reversal and reparameterization.  This is a necessary condition for $\eta'$ to be reversible.}
\end{figure}

Just as in the setting of multiple flow lines, we can compute the conditional law of a counterflow line given the realization of a flow line (or multiple flow lines) within the same geometry.  One case of this which will be particularly important for us is explained in Figure~\ref{fig::duality} --- that the conditional law of $\eta'$ given its left and right boundaries evolves as an $\SLE_{\kappa'}(\tfrac{\kappa'}{2}-4;\tfrac{\kappa'}{2}-4)$ process independently in each of the complementary connected components which lie between its left and right boundaries \cite[\propOB]{MS_IMAG}.  This is sometimes referred to as ``strong duality'' for $\SLE$ (see \cite[Section 8.2] {DUB_PART} for related results).  We remark that $\tfrac{\kappa'}{2}-4$ is the critical value of $\rho$ at which counterflow lines are boundary filling.  When $\rho > \tfrac{\kappa'}{2}-4$, then $\SLE_{\kappa'}(\rho)$ does not fill the boundary and when $\rho \in (-2,\tfrac{\kappa'}{2}-4]$, then $\SLE_{\kappa'}(\rho)$ does fill the boundary.  The situation is analogous for two-sided $\SLE_{\kappa'}(\rho_1;\rho_2)$.

There is an important variant of $\SLE$ duality which allows us to give the law of the outer boundary of the counterflow line $\eta'$ upon hitting any point $z$ on the boundary \cite[\propFP]{MS_IMAG}.  If $z$ is on the left side of $\partial D$, it is given by the flow line of $h$ with angle $-\tfrac{\pi}{2}$ and if $z$ is on the right side of $\partial D$, it is given by the flow of $h$ with angle $\tfrac{\pi}{2}$.   This is explained in Figure~\ref{fig::duality2} in the special case of boundary filling $\SLE_{\kappa'}(\tfrac{\kappa'}{2}-4;\tfrac{\kappa'}{2}-4)$ processes.  This will be particularly important for this article, since it will allow us to describe the geometry of the outer boundary between the set of points that $\eta'$ visits before and after hitting a given boundary point.  Iterating the procedure of decomposing the path into its future and past leads to a new path decomposition of $\SLE_{\kappa'}$ curves.  We remark that this result is closely related to a decomposition of $\SLE_{\kappa'}$ paths into a so-called ``light cone'' of angle restricted $\SLE_\kappa$ trajectories in the same imaginary geometry \cite[Theorem~1.4]{MS_IMAG}.

\section{Proofs}
\label{sec::proofs}

In this section, we will complete the proofs of Theorems~\ref{thm::reversible}--\ref{thm::sle_kappa_rho_non_reversible}.  The strategy for the former two is first to reduce the reversibility of $\SLE_{\kappa'}(\rho_1;\rho_2)$ for $\rho_1,\rho_2 \geq \tfrac{\kappa'}{2}-4$ to the reversibility of $\SLE_{\kappa'}(\tfrac{\kappa'}{2}-4;\tfrac{\kappa'}{2}-4)$ (Lemma~\ref{lem::reduction}).  The main step to establish the reversibility in this special case is Lemma~\ref{lem::reflecting_strip}, which implies that the law of the geometry of the outer boundary of the set of points visited by such a curve before and after hitting a particular boundary point $z$ is invariant under the anti-conformal map which swaps the seed and terminal point but fixes $z$ (see Figure~\ref{fig::duality3}).  This allows us to construct a coupling of two $\SLE_{\kappa'}(\tfrac{\kappa'}{2}-4;\tfrac{\kappa'}{2}-4)$ processes growing in opposite directions whose outer boundary before and after hitting $z$ is the same.  Successively iterating this exploration procedure in the complementary components results in a coupling where one path is almost surely the time-reversal of the other, which completes the proof of reversibility.  The proof of Lemma~\ref{lem::reflecting_strip} will make use of the reversibility of $\SLE_{\kappa}(\rho_1;\rho_2)$ \cite[Theorem~1.1]{MS_IMAG2} and the ``light cone'' characterization of $\SLE_{\kappa'}$ from \cite[Theorem~1.4]{MS_IMAG}.  We will in particular need the variant of $\SLE$ duality described in Figure~\ref{fig::duality2} and in Figure~\ref{fig::duality3} (see \cite[Section 7.4.3]{MS_IMAG}) .

\subsection{Reducing Theorems~\ref{thm::reversible} and~\ref{thm::sle_kappa_rho_reversible} to critical case}
We begin with the reduction of Theorem~\ref{thm::sle_kappa_rho_reversible} to the critical boundary-filling case, which was mostly explained in Figure~\ref{fig::outerboundaryreversibility}.
\begin{lemma}
\label{lem::reduction}
Fix $\rho_1,\rho_2 \geq \tfrac{\kappa'}{2}-4$.  The reversibility of $\SLE_{\kappa'}(\rho_1;\rho_2)$ is equivalent to the reversibility of $\SLE_{\kappa'}(\tfrac{\kappa'}{2}-4;\tfrac{\kappa'}{2}-4)$.
\end{lemma}
\begin{proof}
Suppose that $D$ is a Jordan domain and $x,y \in \partial D$ are distinct.  Assume that $\rho_1,\rho_2 > \tfrac{\kappa'}{2}-4$ and let $\eta'$ be an $\SLE_{\kappa'}(\rho_1;\rho_2)$ from $y$ to $x$.  Let $\psi \colon D \to D$ be an anti-conformal map which swaps $x$ and $y$.  Figure~\ref{fig::duality} implies that the left boundary $\eta_L$ of $\eta'$ is an $\SLE_\kappa(\rho_1^L;\rho_2^L)$ process from $x$ to $y$ for some $\rho_1^L,\rho_2^L > -2$.  Since the time-reversal of $\eta_L$ is an $\SLE_\kappa(\rho_2^L;\rho_1^L)$ process from $y$ to $x$ \cite[Theorem~1.1]{MS_IMAG2}, it follows that $\psi(\eta_L)$ has the law of the left boundary of an $\SLE_{\kappa'}(\rho_1;\rho_2)$ process in $D$ from $y$ to $x$.  Combining Figure~\ref{fig::duality} with Figure~\ref{fig::monotonicity}, we see that the right boundary $\eta_R$ of $\eta'$ conditional on $\eta_L$ is also an $\SLE_\kappa(\rho_1^R;\rho_2^R)$ process for $\rho_1^R,\rho_2^R > -2$ from $x$ to $y$.  Thus \cite[Theorem~1.1]{MS_IMAG2} implies that the time-reversal of $\eta_R$ given $\eta_L$ is an $\SLE_\kappa(\rho_2^R;\rho_1^R)$ process from $y$ to $x$.  Consequently, we have that $\psi(\{\eta_L,\eta_R\})$ has the law of the outer boundary of an $\SLE_{\kappa'}(\rho_1;\rho_2)$ process in $D$ from $y$ to $x$.  By Figure~\ref{fig::duality} (and \cite[\propOB]{MS_IMAG}), we know that the conditional law of $\eta'$ given $\eta_L$ and $\eta_R$ is an $\SLE_{\kappa'}(\tfrac{\kappa'}{2}-4;\tfrac{\kappa'}{2}-4)$ process independently in each of the connected components of $D \setminus (\eta_L \cup \eta_R)$ which lie between $\eta_L$ and $\eta_R$.  This proves the desired equivalence for $\rho_1,\rho_2 > \tfrac{\kappa'}{2}-4$.  The proof is analogous if either $\rho_1 = \tfrac{\kappa'}{2}-4$ or $\rho_2 = \tfrac{\kappa'}{2}-4$.
\end{proof}

\subsection{Main lemma}

For the remainder of this section, we shall make use of the following setup.  Let $\strip = \R \times (0,1)$ be the infinite horizontal strip in $\C$ and let $h$ be a GFF on $\strip$ whose boundary data is as indicated in Figure~\ref{fig::duality3}.  Let $\stripbot$ and $\striptop$ denote the lower and upper boundaries of $\strip$, respectively.  Fix $z \in \stripbot$ and let $\eta_z^1$ be the flow line of $h$ starting at $z$ with angle $\theta_L := (\lambda-\lambda')/\chi = \tfrac{\pi}{2}$ --- this is the flow line of $h+\theta_L \chi$ starting at $z$.  Due to the choice of boundary data, $\eta_z^1$ almost surely hits $\striptop$ (see Figure~\ref{fig::hittingrange}), say at time $\tau_z^1$.  Let $\eta_z^2$ be the flow line of $h$ with angle $\theta_L$ starting from $w = \eta_z^1(\tau_z^1)$ in the left connected component of $\strip \setminus \eta_z^1([0,\tau_z^1])$.  Due to the choice of boundary data, $\eta_z^2$ almost surely hits $\stripbot$ at $z$ (though it will hit $\stripbot$ first in other places, see Figure~\ref{fig::hittingrange}), say at time $\tau_z^2$.

For each $a \in \R$, we let $\reflection_a \colon \C \to \C$ be the reflection of $\C$ about the vertical line through $a$.  We will now prove that the law of $T_z = \{\eta_z^1|_{[0,\tau_z^1]},\eta_z^2|_{[0,\tau_z^2]}\}$, as in Figure~\ref{fig::duality3}, is invariant under $\reflection_z$, up to time-reversal and reparameterization.

\begin{lemma}
\label{lem::reflecting_strip}
The law of $T_z$ defined just above is invariant under $\reflection_z$, up to time-reversal and reparameterization.
\end{lemma}

We note that $\reflection_z$ is the unique anti-conformal map $\strip \to \strip$ which fixes $z$ and swaps $-\infty$ with $+\infty$.  The proof begins with a half-plane version of the construction described in Figure~\ref{fig::duality3} (as would be obtained by ``zooming in near $z$'') which we explain in Figure~\ref{fig::half_plane_setup}.  We then consider a similar construction (using the same instance of the GFF) from a nearby point, as shown in Figure~\ref{fig::half_plane_setup2}.  The result follows from these constructions in a somewhat indirect but rather interesting way.  It builds on time-reversal results for $\SLE_\kappa(\rho_1;\rho_2)$ processes \cite[Theorem~1.1]{MS_IMAG2} (see also \cite{Z_R_KAPPA, DUB_DUAL}) while avoiding additional calculation.

\begin{figure}[!ht]
\begin{center}
\includegraphics[scale=0.85]{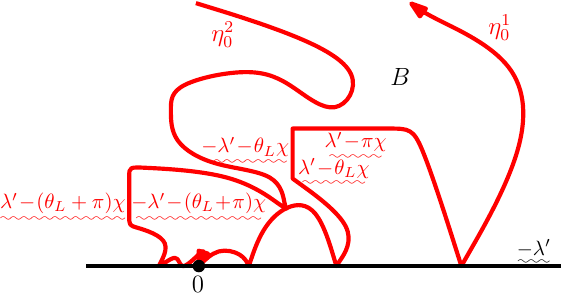}
\caption{\label{fig::half_plane_setup}
We consider the analog of Figure~\ref{fig::duality3} in which $\strip$ is replaced by the entire upper half plane $\h$.  We let $h$ be a GFF on $\h$ with constant boundary data $-\lambda'$ as depicted above.  For $z \in \partial \h$, we let $\eta_z^1$ be the flow line of $h$ with angle $\theta_L$ starting at $z$.  Conditional on $\eta_z^1$, we let $\eta_z^2$ be the flow line of $h$ with angle $\theta_L$ starting at~$\infty$ in the left connected component of $\h \setminus \eta_z^1$.  The particular case $z=0$ is depicted above.  In this case, the symmetry of the law of the pair of paths under reflection about the vertical line through $0$ holds if and only if $c = -\lambda'$, as in the figure (Lemma~\ref{lem::reflecting_half_plane_simple}).  We will extract this from the time-reversal symmetry of $\SLE_\kappa(\rho_1;\rho_2)$ \cite[Theorem~1.1]{MS_IMAG2}.  The area between the pair of paths can be understood as a countable sequence of ``beads''.  Some of these beads have boundaries that intersect the negative real axis, some the positive real axis, some neither axis, and some both axes (see Figure~\ref{fig::hittingrange}).}
\end{center}
\end{figure}

\begin {figure}[!ht]
\begin {center}
\includegraphics[scale=0.85]{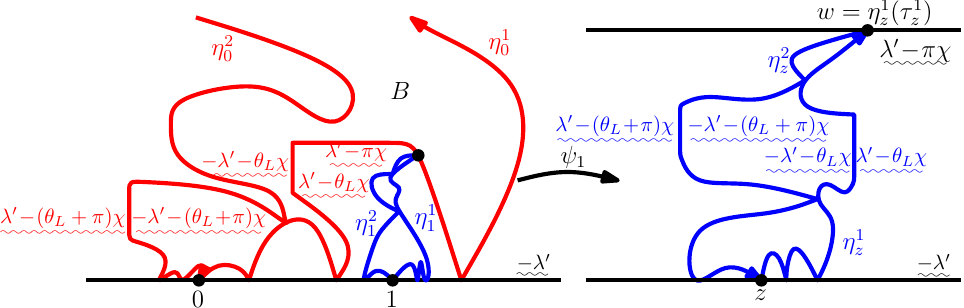}
\caption {\label{fig::half_plane_setup2}
(Continuation of Figure~\ref{fig::half_plane_setup}).  The analogous sequence of beads beginning at~$1$ will at some point merge with the sequence beginning at~$0$ since their boundaries are given by flow lines with the same angle (see \cite[Theorem~1.5]{MS_IMAG}).  Let $B$ be the first bead that belongs to both sequences.  In general, this bead may intersect any subset of the three intervals $(-\infty,0)$, $(0,1)$, and $(1,\infty)$ (see Figure~\ref{fig::hittingrange}).  (In the sketch above, it intersects both $(0,1)$ and $(1,\infty)$.)  Let~$U_1$ be the connected component of~$\h \setminus (\eta_0^1 \cup \eta_0^2)$ which contains~$1$ and let $\psi_1 \colon U_1 \to \strip$ be the conformal map which takes $1$ to $z$, $z \in \stripbot$ fixed, and the left and right most points of $\partial U_1 \cap \R$ to $-\infty$ and $\infty$, respectively.  Then $h \circ \psi_1^{-1} - \chi \arg (\psi_1^{-1})'$ is a GFF on $\strip$ whose boundary data is as depicted on the right side, which is exactly the same as in Figure~\ref{fig::duality3}.}
\end{center}
\end{figure}

\begin{proof}[Proof of Lemma~\ref{lem::reflecting_strip}]
Suppose that $h$ is a GFF on $\h$ with constant boundary data $c=-\lambda'$ as depicted in Figure~\ref{fig::half_plane_setup} and Figure~\ref{fig::half_plane_setup2}.  The main construction in this proof actually makes sense for any $c$ such that $c \leq -\lambda'$ and $c+\theta_L \chi > -\lambda$ (and we will make use of this fact later).  For each $z \in \R$, we let $\eta_z^1$ be the flow line of $h$ starting at $z$ with angle $\theta_L$.  Note that $\eta_z^1$ is an $\SLE_\kappa(\tfrac{-c-\theta_L\chi}{\lambda}-1;\tfrac{c+\theta_L \chi}{\lambda}-1)$ process (see Figure~\ref{fig::conditional_boundary_data}).  Our hypotheses on $c$ imply that both
\[ \frac{-c-\theta_L \chi}{\lambda} - 1 \geq \frac{\kappa}{2}-2  \quad\text{and}\quad \frac{c+\theta_L \chi}{\lambda} -1 \leq -\frac{\kappa}{2} < \frac{\kappa}{2} -2.\]
In the latter inequality, we used that $\kappa \in (2,4)$.  Consequently, $\eta_z^1$ almost surely does not hit $(-\infty,z)$ but almost surely does intersect $(z,\infty)$ (see Figure~\ref{fig::hittingrange} as well as \cite[Remark~5.3]{MS_IMAG}).  Conditionally on~$\eta_z^1$, we let~$\eta_z^2$ be the flow line of $h$ in the left connected component of $\h \setminus \eta_z^1$ starting at~$\infty$ with angle $\theta_L$.  Then $\eta_z^2$ is an $\SLE_\kappa(\kappa-4;\tfrac{c+\theta_L \chi}{\lambda} - 1)$ process in the left connected component of $\h \setminus \eta_z^1$ from $\infty$ to $z$ (see Figure~\ref{fig::conditional_boundary_data} and Figure~\ref{fig::monotonicity}; note also that since $\kappa \in (2,4)$ we have that $\kappa-4 > -2$).  Therefore $\eta_z^2$ almost surely exits $\h$ at $z$ and cannot be continued further (though it may hit $\R$ in $(-\infty,z)$ before exiting; see Figure~\ref{fig::hittingrange}).

We let~$U_0$ be the connected component of $\h \setminus (\eta_1^1 \cup \eta_1^2)$ which contains~$0$ on its boundary.  Similarly, we let~$U_1$ be the connected component of $\h \setminus (\eta_0^1 \cup \eta_0^2)$ which contains~$1$ on its boundary.  Note that the first point on~$\eta_1^1$ that intersects~$\eta_0^1$ is the same as the last point on $\eta_1^2$ that intersects $\eta_0^1$.  Indeed, the reason for this is that $\eta_0^2$ and $\eta_1^2$ agree with each other until the first time that they hit $\eta_0^1 \setminus \eta_1^1$.  This follows because the paths stopped at this time are both given by the flow line from $\infty$ targeted at $0$ of $h$ given $\eta_0^1$ and $\eta_1^1$ in the left connected component of $\h \setminus (\eta_0^1 \cup \eta_1^1)$ (which is the same as the left connected component of $\h \setminus \eta_0^1$ by monotonicity).  Upon hitting $\eta_0^1 \setminus \eta_1^1$, $\eta_0^2$ will continue reflecting off $\eta_0^1$ (and $\partial \h$) until it reaches $0$, while $\eta_1^2$ will merge with $\eta_0^1$.  Indeed, the reason that $\eta_0^1$ and $\eta_1^2$ merge is that they are both flow lines of the common GFF given by the restriction of $h$ to the left connected component of $\h \setminus \eta_1^1$ and have the same angle hence we can apply \cite[Theorem~1.5]{MS_IMAG}.  By their definition, the two paths will agree exactly up to where $\eta_1^1$ hits $\eta_0^1$ for the first time.  Upon hitting $\eta_1^1$ at this point, $\eta_1^2$ will continue bouncing off $\eta_1^1$ and $\partial \h$ until reaching $1$ (as it is the flow line of the restriction of $h$ to the left component of $\h \setminus \eta_1^1$).  Consequently, the restrictions of $\eta_1^1$ and $\eta_1^2$ to $U_1$ meet $\eta_0^1$ at the same point.  By the same argument, the restrictions of $\eta_0^1$ and $\eta_0^2$ to $U_0$ meet $\eta_1^2$ at the same point.  Let $\psi_1 \colon U_1 \to \strip$ be the conformal transformation, as indicated in Figure~\ref{fig::half_plane_setup2}, which takes the left and rightmost points of $\R \cap \partial U_1$ to $-\infty$ and $+\infty$, respectively, and~$1$ to~$z$.  Let~$S_1$ be the image of the restrictions of~$\eta_1^1$ and~$\eta_1^2$ to~$U_1$.  Given~$U_1$, $S_1$ is equal in law to the bead sequence constructed in Figure~\ref{fig::duality3} (see Figure~\ref{fig::half_plane_setup2}).  The same is also true for~$U_0$ when we define~$S_0$ analogously.

Note that $\reflection_{1/2}$ is an anti-conformal automorphism of $\h$ which swaps $0$ and $1$.  Thus $\psi_1 \circ \reflection_{1/2}$ is an anti-conformal map from $\reflection_{1/2}(U_1)$ (which is a neighborhood of $0$) to $\strip$.  Thus, Lemma~\ref{lem::reflecting_strip} is a consequence of Lemma~\ref{lem::reflecting_half_plane}, stated and proved just below (and which uses $c=-\lambda'$).
\end{proof}

Before we state and prove Lemma~\ref{lem::reflecting_half_plane}, we first need the following lemma which gives the reflection invariance of the pair of paths $T_z = \{\eta_z^1,\eta_z^2\}$ for $z \in \partial \h$ (up to a time-reversal and reparameterization of the paths).

\begin{lemma}
\label{lem::reflecting_half_plane_simple}
Suppose that $h$ is a GFF on $\h$ with constant boundary data $c$ with $c \in (-\lambda - \theta_L \chi,-\lambda']$.  For each $z \in \R$, let $\eta_z^1$ be the flow line of $h$ starting at $z$ with angle $\theta_L$.  Conditionally on $\eta_z^1$, let $\eta_z^2$ be the flow line of $h$ in the left connected component of $\h \setminus \eta_z^1$ from $\infty$ with angle $\theta_L$.  Then the law of the pair $\{\eta_z^1 ,\eta_z^2\}$ is invariant under $\reflection_z$ modulo direction reversing reparameterization if and only if $c = -\lambda'$.
\end{lemma}
\begin{proof}
We first suppose that $c=-\lambda'$.  By Figure~\ref{fig::conditional_boundary_data} (see also the beginning of the proof of Lemma~\ref{lem::reflecting_strip}), we know that $\eta_z^1 \sim \SLE_\kappa(\tfrac{\kappa}{2}-2;-\tfrac{\kappa}{2})$ and, conditionally on $\eta_z^1$, $\eta_z^2 \sim \SLE_\kappa(\kappa-4;-\tfrac{\kappa}{2})$ from $\infty$ to $z$ (the $\kappa-4$ force point lies between $\eta_z^2$ and $\eta_z^1$).  By the time-reversal symmetry of $\SLE_\kappa(\rho_1;\rho_2)$ processes \cite[Theorem~1.1]{MS_IMAG2}, this in turn implies that the law of the time-reversal $\CR(\eta_z^2)$ of $\eta_z^2$ given $\eta_z^1$ is an $\SLE_\kappa(-\tfrac{\kappa}{2};\kappa-4)$ process from $z$ to $\infty$ in the left connected component of $\h \setminus \eta_z^1$.  By Figure~\ref{fig::monotonicity}, this in turn implies that $\CR(\eta_z^2)$ is an $\SLE_\kappa(-\tfrac{\kappa}{2};\tfrac{\kappa}{2}-2)$ process from $z$ to $\infty$ in $\h$ and, moreover, the law of $\eta_z^1$ given $\CR(\eta_z^2)$ is an $\SLE_\kappa(\kappa-4;-\tfrac{\kappa}{2})$ process from $z$ to $\infty$ in the right connected component of $\h \setminus \eta_z^1$.  This proves the desired invariance of the law of $\{\eta_z^1 ,\eta_z^2\}$ under $\reflection_z$.  For $c \neq -\lambda'$, the fact that the law of the pair $\{\eta_z^1,\eta_z^2\}$ is not invariant under $\reflection_z$ follows from a similar argument (recall the values of the weights $\rho$ given in the beginning of the proof of Lemma~\ref{lem::reflecting_strip}).
\end{proof}

\begin{lemma}
\label{lem::reflecting_half_plane}
Suppose that $h$ is a GFF on $\h$ with constant boundary data $c$ with $c = -\lambda'$.  For each $z \in \R$, let $\eta_z^1$ be the flow line of $h$ starting at $z$ with angle $\theta_L$.  Conditionally on $\eta_z^1$, let $\eta_z^2$ be the flow line of $h$ in the left connected component of $\h \setminus \eta_z^1$ with angle $\theta_L$ from $\infty$ to $z$.  For any $z,w \in \R$ with $z < w$, the law of $\{\eta_z^1,\eta_z^2,\eta_w^1,\eta_w^2\}$ is invariant under $\reflection_{(z+w)/2}$ (up to time-reversal and reparameterization).
\end{lemma}
\begin{proof}
By rescaling and translating, we may assume without loss of generality that $z=0$ and $w=1$.  For $a \in \{0,1\}$, we let $T_a = \eta_a^1 \cup \eta_a^2$.  We first observe that $T_a$ is independent of $S_{1-a}$ for $a \in \{0,1\}$ (where we recall that $S_a$ and $U_a$ are defined in the proof of Lemma~\ref{lem::reflecting_strip}).  Moreover, $T_a$ and $S_{1-a}$ together determine the paths $(\eta_0^1,\eta_0^2,\eta_1^1,\eta_1^2)$.  This means that we can resample $T_a$ from its original law to obtain $\wt{T}_a$ and then the pair $(\wt{T}_a,S_{1-a})$ determines paths $(\wt{\eta}_0^1,\wt{\eta}_0^2,\wt{\eta}_1^1,\wt{\eta}_1^2)$.  These paths will in general be distinct from $(\eta_0^1,\eta_0^2,\eta_1^1,\eta_1^2)$, however $S_{1-a}$ is fixed under this operation.  (Recall that $S_{1-a}$ is defined in terms of the conformal image of $(\eta_{1-a}^1,\eta_{1-a}^2)$ in $U_{1-a}$ and not just the paths themselves.)  From Lemma~\ref{lem::reflecting_half_plane_simple}, we know that the law of $T_a$ is invariant under $\reflection_a$.

We now consider the following rerandomization transition kernel $\CK$.
\begin{itemize}
\item We pick $j \in \{0,1\}$ with $\p[j=0] = \p[j=1] = \tfrac{1}{2}$.
\item We resample~$T_j$ from its original law to obtain $\wt{T}_j$.  As explained above, $(\wt{T}_j,S_{1-j})$ determines a quadruple of paths $(\wt{\eta}_0^1,\wt{\eta}_0^2,\wt{\eta}_1^1,\wt{\eta}_1^2)$ which has the same law as $(\eta_0^1,\eta_0^2,\eta_1^1,\eta_1^2)$.  These paths in turn determine $(\wt{T}_0,\wt{T}_1,\wt{S}_0,\wt{S}_1)$ where $S_{1-j} = \wt{S}_{1-j}$.
\item We then resample~$\wt{T}_{1-j}$ from its original law to obtain $\wh{T}_{1-j}$.  Then $(\wh{T}_{1-j},\wt{S}_j)$ determines a quadruple of paths $(\wh{\eta}_0^1,\wh{\eta}_0^2,\wh{\eta}_1^1,\wh{\eta}_1^2)$ which has the same law as $(\wt{\eta}_0^1,\wt{\eta}_0^2,\wt{\eta}_1^1,\wt{\eta}_1^2)$.  These paths in turn determine $(\wh{T}_0,\wh{T}_1,\wh{S}_0,\wh{S}_1)$ where $\wh{S}_j = \wt{S}_j$.
\end{itemize}

Let $X_1 = T_0 \cup T_1$.  Clearly, the law of $X_1$ is invariant under $\CK$.  Let $Y_1$ be the image of $X_1$ under $\reflection_{1/2}$.  Since $\CK$ is itself symmetric under $\reflection_{1/2}$, the law of $Y_1$ is also invariant under $\CK$.  We inductively define $X_n$ and $Y_n$ by applying $\CK$ (using the same coin tosses and choices for new $T_0$ and $T_1$ values) to $X_{n-1}$ and $Y_{n-1}$.  Note that each $X_n$ (resp.\ $Y_n$) has the same law as $X_1$ (resp.\ $Y_1$).  Let $K$ be the first time for which, during the rerandomization, we start by resampling $T_0$ and find that the first bead $B$ which is contained in both $T_0$ and $T_1$ intersects both $(0,1)$ and $(1,\infty)$, as depicted in Figure~\ref{fig::half_plane_setup2}, and then we resample $T_1$.  We note that this happens with positive probability in each application of~$\CK$.  Clearly, $X_K = Y_K$ (since they have the same $S_0$ component after resampling $T_0$, and this remains true after the $T_1$ component is resampled for both).  Thus $X_n = Y_n$ for all $n \geq K$ and $K$ is almost surely finite.  Thus $X_1$ and $Y_1$ must indeed have the same law as desired.
\end{proof}

\begin{figure}[ht!]
\begin{center}
\includegraphics[scale=0.85]{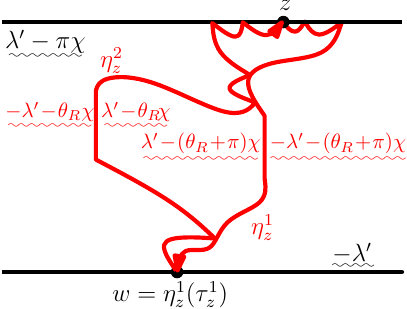}
\caption {\label{fig::strip_setup_left} By symmetry, Figure~\ref{fig::duality3} has an obvious ``upside down'' analog.  Let $h$ be a GFF on the infinite strip $\strip = \R \times (0,1)$ in $\C$ with the boundary data depicted above.  For $z \in \striptop$, we let $\eta_z^1$ be the flow line of $h$ with angle~$\theta_R$.  Then $\eta_z^1$ has to hit the lower boundary $\stripbot$ of $\strip$ (to see this, rotate the picture by $180$ degrees, apply~\eqref{eqn::ac_eq_rel}, and then Figure~\ref{fig::hittingrange}).  Let $w$ be the point where the path $\eta_z^1$ first hits $\stripbot$, say at time $\tau_z^1$.  Conditionally on $\eta_z^1([0,\tau_z^1])$, let $\eta_z^2$ be the flow line of $h$ starting at $w$ with angle $\theta_R$ in the left connected component of $\strip \setminus \eta_z^1([0,\tau_z^1])$.  Then $\eta_z^2$ almost surely exits $\striptop$ at $z$, say at time $\tau_z^2$ (see Figure~\ref{fig::hittingrange}).  In analogy with Lemma~\ref{lem::reflecting_strip}, we have that the joint law of the pair $T_z = \{\eta_z^1|_{[0,\tau_z^1]},\eta_z^2|_{[0,\tau_z^2]}\}$ is invariant under reflecting $\strip$ about the vertical line through $z$ (after time-reversal and reparameterization).}
\end{center}
\end{figure}

\begin{remark}
\label{rem::left_reflecting_strip}
We can define $T_z$ for $z \in \striptop$ analogously and we have a reflection invariance result which is analogous to Lemma~\ref{lem::reflecting_strip}.  This is described in Figure~\ref{fig::strip_setup_left}.
\end{remark}

\subsection{Iteration procedure exhausts curve}

In view of Lemma~\ref{lem::reduction} and Lemma~\ref{lem::reflecting_strip}, we can now complete the proof of Theorem~\ref{thm::reversible} and Theorem~\ref{thm::sle_kappa_rho_reversible}.

\begin{figure}[h!]
\begin{center}
\includegraphics[scale=0.85]{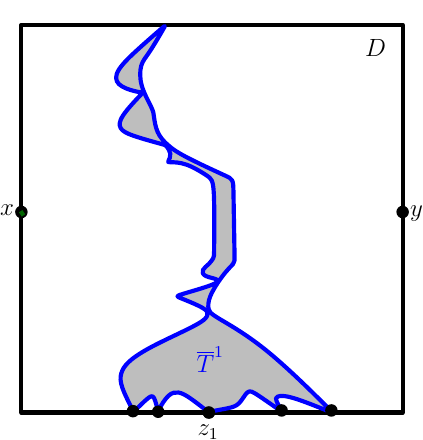}
\end{center}
\caption{\label{fig::iterative_split1}
The first step in the coupling procedure used in the proof of Theorem~\ref{thm::sle_kappa_rho_reversible} for $\rho_1=\rho_2=\tfrac{\kappa'}{2}-4$.  Let $(\eta_{1,-}',\eta_{1,+}')$ be independent $\SLE_{\kappa'}(\tfrac{\kappa'}{2}-4;\tfrac{\kappa'}{2}-4)$ curves in a bounded Jordan domain $D$ with $\eta_{1,-}'$ connecting $y$ with $x$ and $\eta_{1,+}'$ connecting $x$ to $y$, $x,y \in \partial D$ distinct.  Lemma~\ref{lem::reflecting_strip} implies that the law of the set $\ol{T}_{z_1}(\eta_{1,-}')$ which consists of the closure of the set of points which lie between the outer boundaries of $\eta_{1,-}'$ before and after hitting $z_1$ is equal in distribution to the corresponding set $\ol{T}_{z_1}(\eta_{1,+}')$ for $\eta_{1,+}'$.  Therefore we can construct a coupling of $\SLE_{\kappa'}(\tfrac{\kappa'}{2}-4;\tfrac{\kappa'}{2}-4)$ processes $(\eta_{2,-}',\eta_{2,+}')$ such that $\ol{T}^1 = \ol{T}_{z_1}(\eta_{2,-}') = \ol{T}_{z_1}(\eta_{2,+}')$.
}
\end{figure}

\begin{figure}[h!]
\begin{center}
\includegraphics[scale=0.85]{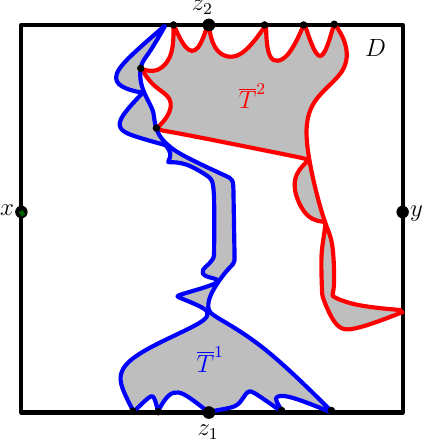}
\end{center}
\caption{\label{fig::iterative_split2}
(Continuation of Figure~\ref{fig::iterative_split1}).  Suppose that $z_2 \in \partial D \setminus \ol{T}^1$.  Then $\eta_{2,-}'$ and $\eta_{2,+}'$ are both $\SLE_{\kappa'}(\tfrac{\kappa'}{2}-4;\tfrac{\kappa'}{2}-4)$ processes in the connected component $C$ of $D \setminus \ol{T}^1$ which contains $z_2$ on its boundary.  Applying Lemma~\ref{lem::reflecting_strip} again thus implies that the law of the set $\ol{T}_{z_2}(\eta_{2,-}')$ which consists of the closure of the set of points which lie between the outer boundaries of $\eta_{2,-}'$ before and after hitting $z_2$ is equal in distribution to the corresponding set $\ol{T}_{z_2}(\eta_{2,+}')$ for $\eta_{2,+}'$.  Therefore we can construct a coupling of $(\eta_{3,-}',\eta_{3,+}')$ such that $\ol{T}^2 = \ol{T}_{z_2}(\eta_{3,-}') = \ol{T}_{z_2}(\eta_{3,+}')$.  The proof proceeds by successively coupling the boundary between the future and past of the two curves until one is almost surely the time-reversal of the other.
}
\end{figure}

\begin{proof}[Proof of Theorem~\ref{thm::sle_kappa_rho_reversible}]
Let $D \subseteq \C$ be a bounded Jordan domain and fix $x,y \in \partial D$ distinct.  We construct a sequence of couplings $(\eta_{k,-}',\eta_{k,+}')$ of $\SLE_{\kappa'}(\tfrac{\kappa'}{2}-4;\tfrac{\kappa'}{2}-4)$ curves on $D$ with $\eta_{k,-}'$ connecting $y$ to $x$ and $\eta_{k,+}'$ connecting $x$ to $y$ as follows.  We take $\eta_{1,-}'$ and $\eta_{1,+}'$ to be independent.  Let $\CD_1 = \{z_{n,1}\}$ be a countable, dense collection of points in $\partial D$ and let $z_1 = z_{1,1}$.  Lemma~\ref{lem::reflecting_strip} implies that the law of $\ol{T}_{z_1}(\eta_{1,-}')$, the closure of the set of points which lie between the outer boundaries of $\eta_{1,-}'$ before and after hitting $z_1$, is equal in law to the corresponding set $\ol{T}_{z_1}(\eta_{1,+}')$ for $\eta_{1,+}'$.  Consequently, there exists a coupling $(\eta_{2,-}',\eta_{2,+}')$ such that $\ol{T}^1 := \ol{T}_{z_1}(\eta_{2,-}') = \ol{T}_{z_1}(\eta_{2,+}')$.  We put $(\eta_{2,-}',\eta_{2,+}')$ onto a common probability space with $(\eta_{1,-}',\eta_{1,+}')$ and take $\eta_{2,-}' = \eta_{1,-}'$; we do not specify how $\eta_{2,+}'$ is coupled with $(\eta_{1,-}',\eta_{1,+}')$.  Note that the order in which $\eta_{2,-}' = \eta_{1,-}'$ visits the connected components of $D \setminus \ol{T}^1$ is the reverse of that of $\eta_{2,+}'$ \cite[\propFP]{MS_IMAG}.  Moreover, the conditional law of $\eta_{2,-}' = \eta_{1,-}'$ given $\ol{T}^1$ is independently an $\SLE_{\kappa'}(\tfrac{\kappa'}{2}-4;\tfrac{\kappa'}{2}-4)$ process in each of these components and likewise for $\eta_{2,+}'$ given $\ol{T}^1$ \cite[\propFP]{MS_IMAG}.

We will now explain how to iterate this procedure.  Let $(d_j)$ be a sequence that traverses $\N \times \N$ in diagonal order, i.e., $d_1 = (1,1)$, $d_2 = (2,1)$, $d_3 = (1,2)$, etc.  Suppose that $k \geq 2$.  We inductively take $\CD_k = \{z_{n,k}\}$ to be a countable, dense subset of $\partial \ol{T}^{k-1}$ and $z_k = z_{d_k}$.  Applying Lemma~\ref{lem::reflecting_strip} again, we know that $\ol{T}_{z_{k}}(\eta_{k,-}')$, the closure of the set of points which lie between the outer boundaries of the set of points visited by $\eta_{k,-}'$ before and after hitting $z_{k}$, is equal in law to $\ol{T}_{z_{k}}(\eta_{k,+}')$, the corresponding set for $\eta_{k,+}'$.  Thus by resampling $\eta_{k,+}'$ in the connected component of $D \setminus \cup_{j=1}^{k-1} \ol{T}^j$ with $z_k$ on its boundary (and leaving the curve otherwise fixed), we can construct a coupling $(\eta_{k+1,-}', \eta_{k+1,+}')$ such that $\ol{T}^{k} := \ol{T}_{z_{k}}(\eta_{k+1,-}') = \ol{T}_{z_{k}}(\eta_{k+1,+}')$ almost surely.  Then the conditional law of $\eta_{k+1,-}'$ and $\eta_{k+1,+}'$ in each of the complementary components of $D \setminus \cup_{j=1}^{k} \ol{T}^j$ is independently that of an $\SLE_{\kappa'}(\tfrac{\kappa'}{2}-4;\tfrac{\kappa'}{2}-4)$ process and $\eta_{k+1,-}'$ visits these connected components in the reverse order of $\eta_{k+1,+}'$ \cite[\propFP]{MS_IMAG}.  We assume that we have put $(\eta_{k,-}',\eta_{k,+}')$ for all $k$ onto a common probability space so that $\eta_{k,-}' = \eta_{1,-}'$ for all $k$.

\begin{figure}[!ht]
\begin{center}
\includegraphics[scale=0.85]{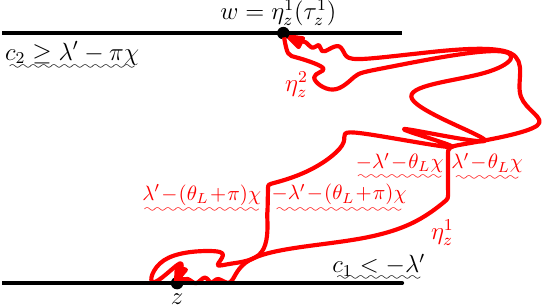}
\caption{\label{fig::non_reversible_strip} The above figure illustrates what happens in the setting of Figure~\ref{fig::duality3} when the constant upper and lower strip boundary data is modified so that the counter flow line $\eta'$ from $+\infty$ to $-\infty$ is still boundary filling, but at least one of the $\rho_i$ (say the one corresponding to the lower boundary) is strictly less than the critical value $\tfrac{\kappa'}{2}-4$.  As in Figure~\ref{fig::duality3}, the paths $\eta^1_z$ and $\eta^2_z$ describe the outer boundary of $\eta'$ before and after hitting $z$.  The law of the pair of paths in a neighborhood of $z$ is absolutely continuous with respect to the law that one would obtain if the upper strip boundary were removed, so that both paths go between $0$ and $\infty$ in $\h$ (and it is not hard to see that the local picture of the pair of paths converges to the half-plane picture upon properly rescaling).  In this case, the ``angle'' between the right path and $(0,\infty)$ is less than that between the left path and $(-\infty,0)$.  Since the opposite is true for the reflected pair of paths (about the vertical line through $z$), the time-reversal of a boundary-filling $\SLE_\kappa(\rho_1;\rho_2)$ can only be an $\SLE_\kappa(\rho_1'; \rho_2')$ (for some $\rho_1', \rho_2'$) if $\rho_1 = \rho_2 = \tfrac{\kappa'}{2}-4$ (recall also Lemma~\ref{lem::reflecting_half_plane_simple}).}
\end{center}
\end{figure}

To complete the proof, we will show that, up to reparameterization, the uniform distance between the time-reversal of $\eta_{k,+}'$ and $\eta_{k,-}'$ converges to $0$ almost surely.  Before we proceed to establish this, let us first explain why this suffices.  By the construction of the coupling, we have that $\eta_{k,-}' = \eta_{1,-}'$ for all $k$.  Consequently, the limit $\lim_{k \to \infty} \eta_{k,-}'$ trivially exists and the law of the limit is that of an $\SLE_{\kappa'}(\tfrac{\kappa'}{2}-4;\tfrac{\kappa'}{2}-4)$ process in $D$ from $y$ to $x$ (as this is the law of $\eta_{1,-}'$).  For each $k$, we also know that $\eta_{k,+}'$ has the law of an $\SLE_{\kappa'}(\tfrac{\kappa'}{2}-4; \tfrac{\kappa'}{2}-4)$ process in~$D$ from~$x$ to~$y$ (modulo time parameterization).  We will show below that the uniform distance between the time-reversal of $\eta_{k,+}'$ and $\eta_{k,-}'$ converges to~$0$ almost surely.  This, in particular, implies that the limit of $\eta_{k,+}'$ as $k \to \infty$ exists almost surely and has the law of an $\SLE_{\kappa'}(\tfrac{\kappa'}{2}-4;\tfrac{\kappa'}{2}-4)$ curve in~$D$ from $x$ to $y$ (modulo reparameterization) as each of the $\eta_{k,+}'$ also have this law.

Let $\CC_k$ be the collection of connected components of $D \setminus \cup_{j=1}^k \ol{T}^j$.  It suffices to show we almost surely have that $\lim_{k  \to \infty} \sup_{C \in \CC_k} \diam(C) = 0$ (the limit exists since $\sup_{C \in \CC_k} \diam(C)$ is decreasing in $k$).  Note that for $C \in \CC_k$ we have that the closure of $(\eta_{1,-}')^{-1}(C)$ is a closed interval; $(\eta_{1,-}')^{-1}(C)$ is itself not an interval since $\eta_{1,-}'$ makes a countable number of excursions from $\partial C$ as it fills it.  Moreover, since $\eta_{1,-}'$ is non-crossing it follows that if $C' \in \CC_k$ is distinct from $C$ then $(\eta_{1,-}')^{-1}(C) \cap (\eta_{1,-}')^{-1}(C') = \emptyset$.  For each $k$, we let
\[ \CI_k = \left\{ \ol{ (\eta_{1,-}')^{-1}(C)} : C \in \CC_k \right\}.\]
If $I,J \in \CI_k$ are distinct they may only intersect at their boundary points.  Moreover, for each $t$ there exists at most one $I \in \CI_k$ such that $t$ is contained in the interior of $I$.  By the continuity of $\eta_{1,-}'$, it suffices to show that we almost surely have 
\begin{equation}
\label{eqn::diam_to_zero}
\lim_{k \to \infty} \sup_{I \in \CC_k} |I| = 0
\end{equation}
where $|I|$ denotes the length of $I$.  We note that~\eqref{eqn::diam_to_zero} does not hold if and only if there exists $\epsilon > 0$ such that for each $k$ there exists $I_k \in \CI_k$ such that $|I_k| \geq \epsilon$.  Since there can only be a finite number of elements of $\CI_k$ with length at least~$\epsilon$, we may further assume that $I_{k+1} \subseteq I_k$ for all $k$.  Then we have that $\cap_k I_k$ has non-empty interior, so it follows that there exists $t$ which is contained in the interior of each $I_k$ (hence cannot be in any other interval $J \in \CI_k$).  That is, \eqref{eqn::diam_to_zero} does not hold if and only if there exists $\epsilon > 0$ and $t$ such that for each $k$ there exists $I_k \in \CI_k$ with $|I_k| \geq \epsilon$ such that $t$ is in the interior of $I_k$.

Assume that there is such an $\epsilon > 0$ and $t$.  Fix $k$ and let $[a,b] \in \CI_k$ be the interval such that $t \in (a,b)$.  Let $C = \eta_{1,-}'(I)$ and note that there exists $z \in (\CD_1 \cup \cdots \cup \CD_k) \cap \partial C$ such that
\[ |\eta_{1,-}'(a) - z| \geq \frac{1}{8}\diam(C) \quad\text{and}\quad |\eta_{1,-}'(b) - z| \geq \frac{1}{8}\diam(C)\]
since $(\CD_1 \cup \cdots \cup \CD_k) \cap \partial C$ is dense in $\partial C$.  Let $\omega$ be the modulus of continuity of $\eta_{1,-}'$.  This implies that, with $\xi = \inf\{t \in [a,b] : \eta_{1,-}'(t) = z\}$ (recall that $\eta_{1,-}'$ almost surely has to hit $z$ since it fills $\partial C$), we have both $\omega(|\xi-a|) \geq \tfrac{1}{8}\diam(C)$ and $\omega(|b - \xi|) \geq \tfrac{1}{8} \diam(C)$.  Using the notation from earlier in the proof, we have that $z = z_{j(k)}$ for some $j(k) > k$.  Let $[a',b'] = I_{j(k)}$.  By the construction, we have that $[a',b']$ is contained in either $[a,\xi]$ or $[\xi,b]$.  In the former case, we have
\begin{align*}
&|a' - b'|
    \leq  |a - b| - |b - \xi|\\
	\leq& |a-b| - \omega^{-1}(\tfrac{1}{8}\diam(C)).
\end{align*}
Similarly, in the latter case,
\begin{align*}
 & |a' - b'|
    \leq  |a - b| - |a - \xi|\\
	\leq& |a - b| - \omega^{-1}(\tfrac{1}{8}\diam(C)).
\end{align*}
This leads to a contradiction if $\lim_{k \to \infty} |I_k| > 0$.
\end{proof}

\subsection{Non-critical boundary-filling paths}

We now prove Theorem~\ref{thm::sle_kappa_rho_non_reversible}, which states that $\SLE_{\kappa'}(\rho_1;\rho_2)$ does not have time-reversal symmetry when $\min(\rho_1,\rho_2) < \tfrac{\kappa'}{2}-4$.

\begin{proof}[Proof of Theorem~\ref{thm::sle_kappa_rho_non_reversible}]
We may assume without loss of generality that $\rho_1 \in (-2,\tfrac{\kappa'}{2}-4)$.  This case is treated in Figure~\ref{fig::non_reversible_strip}, together with the discussion of the half-plane problem for general constant boundary values $c$ that was given in the proofs of Lemma~\ref{lem::reflecting_strip} and Lemma~\ref{lem::reflecting_half_plane_simple}.
\end{proof}

\section{Couplings}
\label{sec::couplings}

\begin{figure}[ht!]
\begin{center}
\includegraphics[scale=0.85]{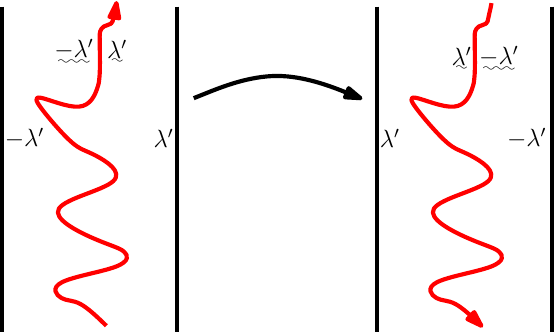}
\end{center}
\caption{\label{fig::simpleswap}
Consider a GFF $h$ on the infinite vertical strip $[-1,1] \times \R$ whose boundary values are depicted on the left side above.  The flow line $\eta$ of $h$ from the bottom to the top is an $\SLE_\kappa(\rho^L;\rho^R)$ process with $\rho^L = \rho^R = \tfrac{\kappa}{2}-2$.  To go from the left figure to the right figure, add the constant function $2\lambda'$ to left side of the strip minus the path and $-2\lambda'$ to the right side to obtain a new field $\wt h$ with the boundary conditions shown on the right.  By the reversibility of $\SLE_\kappa(\rho^L;\rho^R)$ processes \cite[Theorem~1.1]{MS_IMAG2}, $\wt h$ is a GFF with the boundary data indicated on the right side and the time-reversal of $\eta$ is a flow line of $\wt h$ from the top to the bottom.  Note that $h - \wt h$ is piecewise constant, equal to $-2\lambda'$ to the left of the path and $2 \lambda'$ to the right.  (We have not defined the difference on the almost surely zero-Lebesgue measure path $\eta$, but this does not effect the interpretation of this difference as a random distribution.)
}
\end{figure}

\begin{figure}[ht!]
\begin{center}
\includegraphics[scale=0.85]{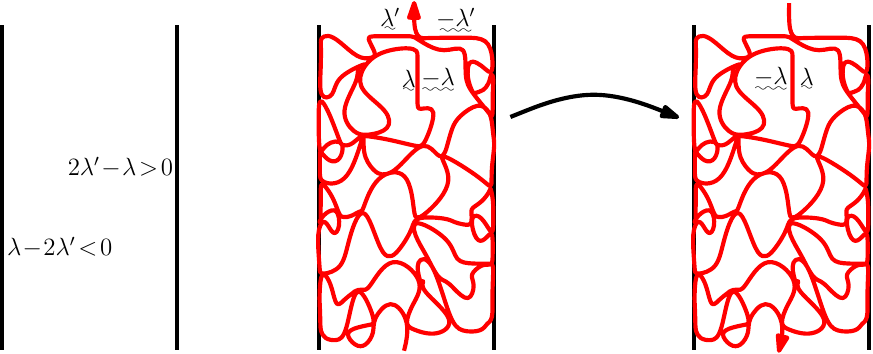}
\end{center}
\caption{\label{fig::simplecounterflowswap}
Consider a GFF $h$ on the infinite vertical strip $[-1,1] \times \R$ whose boundary values are depicted on the left above.  The counterflow line $\eta'$ of $h$ from the bottom to the top depicted in the middle above is an $\SLE_\kappa(\rho^L;\rho^R)$ process with $\rho^L = \rho^R = \tfrac{\kappa'}{2}-4$.  To see this, we note that $2\lambda'-\lambda = \lambda' -\tfrac{\pi}{2} \chi$ and $\lambda-2\lambda' = -\lambda' + \tfrac{\pi}{2}\chi$.  Consequently, if $\psi$ is the conformal map which rotates the strip $90$ degrees in the counterclockwise direction, the coordinate change formula~\eqref{eqn::ac_eq_rel} implies that the boundary conditions of the GFF $h \circ \psi^{-1} - \chi \arg (\psi^{-1})'$ agree with those of the GFF on the horizontal strip as depicted in Figure~\ref{fig::duality2}.   This is a path that is boundary filling but not space filling and it divides the strip into countably many regions that lie ``left'' of the path and countably many that lie ``right'' of the path.  As in Figure~\ref{fig::simpleswap}, we can go from the middle to the right figure by adding $-2\lambda$ to the left side of the path and $2\lambda$ to the right side of the path.  By the reversibility of $\SLE_\kappa(\rho^L;\rho^R)$ processes, the resulting field $\wt h$ is a GFF with the boundary data equal to $-1$ times the boundary data in the left figure, and the time-reversal of $\eta'$ is the flow line of $\wt h$ from the top to the bottom.  Note that $h - \wt h$ is piecewise constant, equal to $2\lambda$ on the left of the path and $-2 \lambda$ on the right.}
\end{figure}


One interesting aspect of time-reversal theory is that it allows us to couple two Gaussian free fields $h$ and $\wt h$ with different boundary conditions in such a way that their difference is almost surely piecewise harmonic.  In fact, for special choices of boundary conditions, one can arrange so that this difference is almost surely piecewise constant, with the boundary between constant regions given by an appropriate $\SLE_\kappa(\rho^L;\rho^R)$ curve.  We illustrate this principle for flow lines in Figure~\ref{fig::simpleswap}, which we recall from \cite{MS_IMAG2}.

The expectation field $E(z) := (\E(h-\wt h))(z)$ where $h$ and $\wt{h}$ are as in Figure~\ref{fig::simpleswap} is a linear function equal to $-2\lambda'$ on the left boundary of the strip and $2\lambda'$ on the right boundary.  If we write $P_L(z)$ for the probability that $z$ is to the left of the curve and $P_R(z) = 1-P_L(z)$ for the probability that $z$ is to the right, then $E(z) = -2\lambda'P_L(z) + 2\lambda' P_R(z)$, which implies that $P_L(z)$ is a linear function equal to $1$ on the left boundary of the strip and $0$ on the right boundary.

In light of the results of this paper, one can produce a variant of Figure~\ref{fig::simpleswap} involving counterflow lines when $\kappa \in (2,4)$ so that $\kappa' \in (4,8)$.  We illustrate this in Figure~\ref{fig::simplecounterflowswap}.  In this case, the expectation field $E(z) := (\E(h-\wt h))(z)$ is a linear function equal to $2(\lambda - 2\lambda') < 0$ on the left and $2(2\lambda' - \lambda) > 0$ on the right.  If we define $P_L(z)$ and $P_R(z)$ as above, then in this setting we have $E(z) = 2\lambda P_L(z) - 2\lambda P_R(z)$.  This implies that $P_L$ is a linear function equal to
$\frac{2(\lambda - 2\lambda') + 2\lambda}{4 \lambda} = 1 - \lambda'/\lambda = (4- \kappa)/4 \in (0, 1/2)$ on the left side and $1$ minus this value, which is $\kappa/4 \in (1/2,1)$, on the right side.

At first glance it is counterintuitive that points near the left boundary are more likely to be on the right side of the path.  This is the opposite of what we saw in Figure~\ref{fig::simpleswap}.  To get some intuition about this, consider the extreme case that $\kappa'$ and $\kappa$ are very close to $4$.  In this case, the $\rho^L$ and $\rho^R$ values in Figure~\ref{fig::simplecounterflowswap} are very close to $-2$, which means, intuitively, that when $\eta(t)$ is on the left side of the strip, it traces very closely along a long segment of the left boundary before (at some point) switching over to the right side and tracing a long segment of that boundary, etc.  Given this intuition, it is not so surprising that points near the left boundary are more likely to be to the right of the path.  When $\kappa'$ is close to $8$ (so that the counterflow line is close to being space filling) $P_L$ and $P_R$ are close to the constant function $\tfrac{1}{2}$.  Here the intuition is that the path is likely to get very near to any given point $z$, and once it gets near it has a roughly equal chance of passing $z$ to the left or right.

\bibliographystyle{hmralphaabbrv}
\bibliography{reversibility}

\bigskip

\filbreak
\begingroup
\small
\parindent=0pt

\bigskip
\vtop{
\hsize=5.3in
Microsoft Research\\
One Microsoft Way\\
Redmond, WA, USA }

\bigskip
\vtop{
\hsize=5.3in
Department of Mathematics\\
Massachusetts Institute of Technology\\
Cambridge, MA, USA } \endgroup \filbreak

\end{document}